\documentclass[11pt]{article}
\usepackage[all,2cell]{xy} \UseAllTwocells \SilentMatrices
\usepackage{latexsym,amsfonts,amssymb}
\usepackage{amsmath,amsthm,amscd}
\usepackage{hyperref,psfrag}
\usepackage{graphicx}
\usepackage{epsfig}
\usepackage{cite, tocvsec2}
\usepackage{amsfonts, amstext, mathrsfs}
\usepackage{amsopn,amscd,xspace,enumerate,multicol}
\usepackage{xspace}

\usepackage{a4wide}

\normalfont\upshape

\usepackage{fancyheadings}
\newcommand\scalemath[2]{\scalebox{#1}{\mbox{\ensuremath{\displaystyle #2}}}}

\newcommand\lie[1]{\mathfrak{#1}}
\newcommand{\lb}{\lie{b}}
\newcommand{\lu}{\lie{u}}
\newcommand{\n}{\lie{n}}

\newcommand{\g}{\lie{g}}
\newcommand{\s}{\lie{sl}}
\newcommand{\h}{\lie{h}}

\newcommand{\C}{\mathbb{C}}

\newcommand{\Z}{\mathbb{Z}}
\newcommand{\X}{\mathbb{X}}
\newcommand{\Y}{\mathbb{Y}}
\newcommand{\N}{\mathbb{N}}
\newcommand{\Nt}{{ \widetilde{\mathcal{N}} }}
\newcommand{\Ox}{\mathcal{O}}
\newcommand{\Qx}{\mathcal{Q}}
\newcommand{\F}{\mathcal{F}}
\newcommand{\Lx}{\mathcal{L}}
\newcommand\op[1]{{\rm{#1}}}
\newcommand\mH{\mathrm{H}}
\newcommand\lra{\longrightarrow}

\newcommand\pr{\mathrm{pr}}

\newcommand\ad{\mathrm{ad}}
\newcommand\al{\alpha}
\newcommand{\Ul}{{\mathsf U}}
\newcommand{\ul}{u_q({\mathfrak g})}
\newcommand{\zl}{{\mathsf z }}
\newcommand{\A}{ \mathcal{A} }
\newcommand{\m}{\lie{m}}

\theoremstyle{theorem}
\newtheorem{theorem}{Theorem}[section]
\newtheorem{corollary}[theorem]{Corollary}
\newtheorem{conjecture}[theorem]{Conjecture}
\newtheorem{lemma}[theorem]{Lemma}
\newtheorem{proposition}[theorem]{Proposition}

\theoremstyle{definition}
\newtheorem{definition}[theorem]{Definition}
\newtheorem{example}[theorem]{Example}
\newtheorem{remark}[theorem]{Remark}

\numberwithin{equation}{section}

\psfrag{s1}{$s_1$}
\psfrag{s2}{$s_2$}

\title{The center of small quantum groups I: the principal block in type A}
\author{Anna Lachowska, You Qi}
\date{\today}

\begin{document}

\maketitle
\tableofcontents

\begin{abstract}
We develop an elementary algebraic method to compute 
the center of the principal block of a small quantum group associated with a complex semisimple Lie algebra at a root of unity. The cases of $\mathfrak{sl}_3$ and $\mathfrak{sl}_4$ are computed explicitly.  This allows us to formulate the conjecture that, as a bigraded vector space, the center of 
a regular block of the small quantum $\mathfrak{sl}_m$ at a root of unity is isomorphic to Haiman's  
diagonal coinvariant algebra for the symmetric group $S_{m}$.
\end{abstract}

\section{Introduction}
\subsection{Motivation} 
The problem of determining the structure of the centers of small quantum groups at roots of unity has
a long history. Even before the small quantum group $u_q(\g)$ was defined by Lusztig~\cite{LusfdHopf} for a semisimple Lie algebra $\g$ and an $l$-th root of unity $q$, a similar problem was considered for algebraic groups over fields of positive characteristic (see, e.g. \cite{Haboush}, which studies the algebra of distributions supported at $1$ of a reductive algebraic group).  The same question is closely related to the problem of finding the center of the restricted enveloping algebra of a reductive algebraic group over a field of positive characteristic. 
In both quantum and modular cases the objects under consideration are finite-dimensional
Hopf algebras whose structure are determined by a finite root system and an integer or a prime parameter.
By the work of Andersen, Jantzen and Soergel~\cite{AJS}, the principal blocks of both algebras are Morita equivalent to the same algebra (up to a base field change, with some restrictions on $l$), meaning that
an answer for the structure of the center for one of them translates to the other.

In addition to the original motivation for the study of the center of $u_q(\g)$, based on the connection with
the structure theory of algebraic groups over fields of positive characteristic, other potential applications
should be mentioned.  One is suggested by the equivalences of categories between the representations
of quantum groups at roots of unity and affine Lie algebras. The G-invariant part of 
the center of the small quantum group is contained in the center
of the big quantum group, which has a representation category equivalent to that of an affine Lie algebra at a negative level by 
\cite{KaLuIII, KaLuIV} and \cite{ABG}. Another possible application comes from the theory of quantum topological invariants of
$3$-manifolds \cite{RT2, KauRad, Lyu}.  It is known \cite{Hen} that a family of quantum invariants,
including the Reshetikhin-Turaev and Hennings-Kauffman-Radford invariants, is parametrized by
certain special elements of the center of $u_q(\g)$.  Another direction has been suggested in a series of papers
 studying logarithmic conformal field theories (see, e.g., \cite{FGST}). In case of $\g = \s_2$, the small
 quantum group $u_q(\s_2)$ and a certain $W$-algebra act on the same vertex operator algebra, and their
 actions centralize each other. The expectation is that this observation extends to higher rank, and that there is a
 strong relation between the Hochschild cohomology of the categories involved and the centers of the two algebras.

Despite the fact that $u_q(\g)$ is a finite dimensional algebra over $\C$ that has been
studied extensively for over 20 years,
the structure and even the dimension of its center has remained unknown, except in the case of
$\g = \s_2$. In the latter case the answer was first found in \cite{Ker}: the dimension of the center
of $u_q(\s_2)$, with $q$ a primitive root of unity of odd degree $l \geq 3$, equals $\frac{3l-1}{2}$, which
is unexpectedly large (the number of inequivalent irreducible representations of $u_q(\s_2)$ is $l$).
For higher rank, \cite{La} contains a description of the largest known central subalgebra, and
provides a lower bound for the dimension of the center. In particular, the principal block of the center
of $u_q(\g)$, whose structure is independent of $l$ by \cite{AJS}, contains a subalgebra of
dimension $2|W|-1$, where $W$ is the Weyl group associated with $\g$. For $\g = \s_2$, this 
subalgebra coincides with the whole center. For higher rank, it was expected from the
beginning that the dimension of the center of the principal block of $u_q(\g)$ should be larger than $2|W|-1$.
Building on \cite{ABG}, a description of the Hochschild cohomology of the principal block of $u_q(\g)$
is given in \cite{BeLa} for any semisimple $\g$ in terms of the cohomology of certain polyvector fields over the Springer resolution. The previously known central
subalgebra is clearly visible in this framework; however, it does not provide an immediate answer for the combinatorial structure or even the dimension of the center of the principal block for higher rank.

\subsection{Summary}
In the present paper we develop a method to compute explicitly the sheaf cohomology groups involved in \cite{BeLa}. We carry out a detailed computation for $\g = \mathfrak{sl}_3$, and present the basic steps and the result of a computation for $\g =\mathfrak{sl}_4$. This allows us to formulate an intriguing conjecture (Conjecture \ref{conj-DC}) for the structure of the principal block of the center of the small quantum group in type $A$.   In a sequel \cite{LQ2} we will give a parallel discussion of the singular blocks for $\mathfrak{sl}_3$ and more generally in type A. Further computations will be performed in subsequent works in order to formulate similar conjectures for other Lie types. 

The organization of the paper is as follows. 
In Section \ref{sec-Springer} we present a method for computation of the dimension of the principal block of the center of $u_q(\g)$.  
Let $G$ be a complex semisimple Lie group, and $B$ be a fixed Borel subgroup in $G$. 
We start by recalling the main object of study, the Springer variety $\Nt$ as the cotangent bundle of the flag variety $G/B $, and its connection to the principal block of small quantum groups \cite{ABG, BeLa}. The main result in \cite{BeLa} (Theorem \ref{Hoh}) allows us to reduce the center computation for small quantum groups to the cohomology of the poly-tangent bundle $\wedge^\bullet T\Nt$ on the Springer variety. By pushing forward the poly-tangent bundle onto the flag variety $G/B$ along the canonical projection map $\mathrm{pr}: \Nt\lra G/B$, one obtains a family of equivariant vector bundles on $G/B$. The coherent cohomology of any equivariant vector bundle $G\times_B E$ on $G/B$ can be computed via Bott's classical result (Theorem \ref{Bott_rel_lie}), relating them to more algebraically approachable (relative) Lie algebra cohomology groups with coefficients in $E$. Our first main result in this Section is an explicit description of the equivariant structure of the pushforward sheaves $\mathrm{pr}_*(\wedge^\bullet T\Nt)$. The answer is formulated in Theorem \ref{thm-equ-structure-TN} and Corollary \ref{cor-exterior-product}. This allows us, in principle, to compute the center of the principal block of the small quantum groups via relative Lie algebra cohomology. The second ingredient of our approach consists of a simplification of Bott's method by using the Bernstein-Gelfand-Gelfand resolution of a finite-dimensional simple $\g$-module that reduces the relative Lie algebra cohomology to the combinatorics of $E$ (Proposition \ref{rel_lie_BGG}). The combination of these two ingredients allows us to obtain an explicit algorithmic method to compute the principal block of the center of small quantum groups which, in small rank cases, can be performed by hand.

In Section \ref{sec-sl3-center}, we apply the general machinery developed in Section \ref{sec-Springer} to the particular case of $u_q(\mathfrak{sl}_3)$. After fixing the specific notation for this case in Section \ref{sec-ntn}, we compute the cohomology of two auxiliary vector bundles in Section \ref{sec-some-sheaves}, illustrating both the geometric and algebraic methods involved. 
In Section \ref{sec-sl3}, the zeroth Hochschild cohomology of $\Nt$ is computed explicitly, and the results are tabulated in Theorem \ref{thm-sl3}. In particular, the dimension of the principal block of the center for $\mathfrak{sl}_3$ is $16$, considerably greater than the dimension of the previously known subalgebra $2|W|-1 = 11$. Our key observation is that the bigraded components of the center fit into a ``formal Hodge diamond," which is isomorphic to Haiman's bigraded diagonal coinvariant algebra \cite{Hai} for $\mathfrak{sl}_3$ up to an overall bigrading transformation.
This remarkable correspondence between the center of the principal block of $u_q(\g)$ and Haiman's diagonal coinvariant algebra is, in retrospect, also confirmed for $\g = \mathfrak{sl}_2$ via the work of \cite{Ker}. 

The main goal of Section \ref{sec-center-symmetry} is to formulate a conjecture generalizing the $u_q(\mathfrak{sl}_3)$ case to a general finite type $A$ situation.  Our first observation is that there exists an $\mathfrak{sl}_2$ action on the Hochschild cohomology groups of $\Nt$ for any $\g$ (Theorem \ref{thm-sl2-action}), which resembles the usual $\mathfrak{sl}_2$ action on the Dolbeault cohomology ring of a smooth compact K\"{a}hler manifold. Since $\Nt$ is holomorphic symplectic, such an action is generated by wedging with the Poisson bivector field $\tau$ and contracting with the canonical symplectic form $\omega$ on $\Nt$. Consequently, the center $\zl_0(u_q(\g))$ contains a large subalgebra $\mathrm{\tau C_\g}$ 
generated by the cohomology ring of $G/B$, which is isomorphic to the coinvariant algebra $\mathrm{C_\g}$, and the Poisson bivector field $\tau$ (Corollary \ref{cor-big-subalg}). This subalgebra contains the previously known largest subalgebra found in \cite{La}. Thus in the case of $u_q(\mathfrak{sl}_2)$, $\mathrm{\tau C_{\mathfrak{sl}_2}}$ agrees with the entire principal block center, while our computation shows that, for $u_q(\mathfrak{sl}_3)$, it has codimension $1$ in $\zl_0(u_q(\mathfrak{sl}_3))$. However, as the main Conjecture \ref{conj-DC} would imply, the codimension of the subalgebra $\mathrm{\tau C_{\mathfrak{sl}_m}}\subset \zl_0$ grows exponentially in $m$ for $u_q(\mathfrak{sl}_m)$. Finally, in Section \ref{sec-further-evidence}
 we compute the bigraded dimension of the center in the case $u_q(\mathfrak{sl}_4)$ (see Theorem 
 \ref{conj-sl4}). The computation is based on the algorithm developed in Section \ref{sec-Springer}; some of the entries are computed using a Python code based on this algorithm. The obtained result confirms 
 Conjecture \ref{conj-DC} in the case of $u_q(\mathfrak{sl}_4)$.

We believe that the similarity between the two canonically defined objects associated with $\mathfrak{sl}_m$, namely the principal block of the center $\zl_0(u_q(\mathfrak{sl}_m))$\footnote{By the work of Andersen-Jantzen-Soergel \cite{AJS}, this commutative $\C$-algebra $\zl_0(u_q(\mathfrak{sl}_m))$ is independent of the order $l$ of the root of unity, if $l$ is greater than  the Coxeter number of $\g$.}, and Haiman's diagonal coinvariant algebra $\mathrm{DC}_m$, is not merely a coincidence for $m=2, 3$ and $4$. To reveal the algebro-geometric and representation theoretical reasons behind this remarkable similarity will be the goal of our subsequent work. 

To conclude this summary, let us emphasize that, through the Frenkel-Gaitsgory derived equivalence \footnote{The derived version is already enough to imply the isomorphism of centers. For a stronger (conjectural) abelian equivalence, see \cite{BezrLin}.} between the principal 
block of the small quantum group and a category of certain modules over the affine Lie algebra at a critical level  \cite{FG}, 
the validity of Conjecture \ref{conj-DC} would also shed new light on understanding of the principal block of $\widehat{\mathfrak{sl}}_m$ at the critical level. 

\subsection{Some further questions}
The current work is only an initial step in the investigation of the center for small quantum groups. Here we briefly outline some future directions we plan to pursue.

In the sequel \cite{LQ2} to this paper, we consider the same problem for singular blocks of small quantum groups. We will first formulate a generalization of the result in \cite{BeLa} relating singular block centers to the zeroth Hochschild cohomology of parabolic Springer varieties, and then compute the center of the singular blocks for  $\g = \mathfrak{sl}_3$ together with some other examples via the method developed in this paper.

Finding an algebraic interpretation of the newly discovered central elements for small quantum groups constitutes another important problem. The previously known largest central subalgebra of the regular block of the center identified in \cite{BeLa} 
together with its analogs in the singular blocks corresponds to a certain subspace of tracelike functionals described in 
 \cite{La}. We would like to know which trace-like functionals correspond to the newly found generators in the zeroth Hochschild cohomology ring in both regular and singular cases.  In particular, the emergence of the $\s_2$ action on the center demands an interpretation in the framework of the representation theory of the small quantum group.   We plan to address this question in the subsequent works. 

The same problems for other Lie types will also be studied in upcoming works. The method developed in this paper is adaptable to machine computation, and pursuing this path will allow us to formulate similar conjectures for other Lie types. For instance, it would be interesting to get an explicit answer for type $B_4$, the lowest-rank example where the diagonal coinvariant algebra in type $B$ differs from its canonical quotient \cite{Hai, Gor, CheDiagonal}. Because of the large dimensions of the vector bundles involved, this computation is not easily accessible by hand.

 We will continue working on obtaining further evidence and, hopefully, a proof for the formulated conjectures on the structure of the center of small quantum groups. Furthermore, it would be interesting to find a connection between centers of small quantum groups and categorified small quantum groups, as initiated in \cite{KQ, EQ1}, and to understand how the Hecke categories in \cite{RW}  are related to the center of small quantum group, which we also plan to pursue in the future.

\paragraph{Acknowledgments.} The authors would like to thank Roman Bezrukavnikov, Ivan Cherednik, A.~Johan de Jong, Dennis Gaitsgory, Jiuzu Hong, Steve Jackson, Mikhail Khovanov, Peng Shan, Geordie Williamson and Chenyang Xu for helpful discussions. We are thankful to Anton Mellit for giving a talk in November 2015 at EPFL that attracted our attention to Haiman's diagonal coinvariant algebra, which resulted in a decisive advancement in our project. We are immensely grateful 
to Bryan Ford for writing a Python code implementing our algorithm in the case $\g = \mathfrak{sl}_4$, which allows 
us to obtain a complete answer in this example.   
 Special thanks go to Igor Frenkel for his constant encouragement and for bringing to the authors' attention some potential physical implications of the results.

\section{The Springer resolution and the tangent bundle}\label{sec-Springer}
In this section we develop a method that will allow us to compute the center of the principal block of the small quantum $\mathfrak{sl}_3 $, and provide an algorithm for performing similar computations in general. In particular, we analyze the structure of the tangent bundle of the Springer resolution $\Nt$ and its exterior powers relative to the flag manifold. 

\subsection{Elements}\label{sec-elements}
\paragraph{Notation.}In this section, $G$ is a reductive group over an algebraically closed field of characteristic zero\footnote{As an abuse of notation, we will always use $\mathbb{C}$ to stand for this ground field}, and $B$ is a fixed Borel subgroup of $G$. The unipotent subgroup $[B,B]$ and Cartan subgroup $B/[B,B]$ will be written as $N$ and $H$ respectively, so that $B=HN$. The opposite unipotent group to $N$ will be denoted by $U$.

Let $X:=G/B$ be the flag variety associated with $G$.

The (complex) Lie algebras for the corresponding groups will be denoted by the lower case Gothic letters: 
\begin{equation}\label{eqn-lie-algebras}
\g:=\mathrm{Lie}(G), \quad \lb:=\mathrm{Lie}(B), \quad \n:=\mathrm{Lie}(N), \quad \h:=\mathrm{Lie}(H), \quad \lu:=\mathrm{Lie}(U).
\end{equation}
We have $\n=[\lb,\lb]$, and $\lb \cong \h\oplus \n$.

The Lie groups act on their Lie algebras by the adjoint representation. Choose a non-degenerate $G$-invariant bilinear form on $\g$ (e.g.~the Killing form if $\g$ is simple). Under this paring we have 
$\n \cong \lu^*$, which is, in fact, an isomorphism of $B$-modules. Likewise, we have $\h\cong \h^*$ as $G$-modules. 

Later we will be concerned with \emph{($G$-)equivariant vector bundles} over $X$ and their associated (sheaf) cohomology groups. Such bundles are necessarily of the form $G\times_B V$, where $V$ is a $B$-representation. Their cohomology groups then admit a natural $G$-action. 
Some vector bundles over $X$ can be upgraded (not uniquely) to equivariant bundles 
(see \ref{eg-equiv-bundle} for a non-uniqueness example). 
For instance, we will identify
\[
TX\cong G\times_B(\g/\mathfrak{b}), \quad \quad T^*X\cong G\times_B(\g/\mathfrak{b})^*.
\]
Clearly, as $B$-modules, we have $\g/\mathfrak{b}\cong \lu$. Moreover, via the $G$-equivariant bilinear form on $\g$, we identify  $(\g/\mathfrak{b})^*\cong \n$, so that
\begin{equation}\label{eqn-tan-cotan-as-equiv-bdl}
TX\cong G\times_B \lu, \quad \quad T^*X\cong G\times_B \n.
\end{equation}

The group $G$ acts on $X$ on the left, and $B$ is the stabilizer subgroup of the identity coset. If $G\times_B V$ is an equivariant vector bundle on $X$, then $B$ acts naturally on the fiber over the identity coset $eB\in X$, which is no other than the vector space $V$ regarded as a variety. It follows that, if $V$ is an indecomposable representation of $B$, then the bundle $G\times_B V$ can not decompose into a nontrivial direct sum of equivariant subbundles. This useful property is a special case of the following well known lemma.

\begin{lemma}\label{lemma-equivalent-equivariant-coh}
Let $G$ be an algebraic group and $P$ be a Zariski closed subgroup. Suppose $V$ is a linear $P$-representation. Then the category of $G$-equivariant coherent sheaves on the variety $G\times_P V$ is equivalent to the category of $P$-equivariant sheaves on $V$.
\end{lemma}
\begin{proof}The equivalence is provided by induction and restriction to the fiber over the identity coset of the canonical projection map
$\mathrm{pr}:G\times_PV\lra G/P$.
\end{proof}

\begin{example}\label{eg-equiv-bundle}
Consider the adjoint bundle $G\times_B\g$ on $X$. Let $\nu$ be the canonical $G$-equivariant projection map $X=G/B\longrightarrow G/G $. Since $\g$ is a $G$-representation, this bundle is the equivariant pull-back of the $G$-bundle $\g$ over the point $G/G$:
\[
G\times_B\g\cong \nu^*(G\times_G \g).
\]
Therefore, as a non-equivariant bundle,
\[
G\times_B \g\cong G\times_B (\mathbb{C}^{\oplus \mathrm{dim}\g})
\]
is a trivial vector bundle on $X$. But they are not isomorphic as equivariant bundles, for $\g$ is an indecomposable $B$-module generated by a highest weight vector. Upon taking global sections, the left-hand side gives us the adjoint representation of $G$, while the right-hand side results in $\mathrm{dim}\g$ copies of the trivial $G$-representation.
\end{example}

Let $\mathcal{N}$ be the nilpotent cone of $\g$, which consists of elements in $\g$ that can be conjugated inside $\n$ under the adjoint action of $G$. The \emph{Springer resolution} of $\mathcal{N}$, denoted  by
\begin{equation}
\pi:\Nt\longrightarrow \mathcal{N}
\end{equation}
is a resolution of the singularity for $\mathcal{N}$. (See \cite[Chapter III]{CG} for the standard facts about the Springer resolution). As an algebraic variety, $\Nt$ can be identified with the cotangent bundle to $X$:
\begin{equation}\label{eqn-cotan-bundle-as-rsltn}
\Nt\cong T^*X\cong G\times_B \n.
\end{equation}
Elements in $\Nt$ are thus given by pairs $(g,x)$, where $g\in G$ and $x\in \n$, subject to the identification $(g,x)=(gb^{-1}, \mathrm{Ad}_b(x)))$. Let $\op{pr}: T^*X \lra X$ be the canonical projection map that sends the equivalence 
class of the pair $(g,x)$ to the coset $gB$. It is evidently $G$-equivariant. 

Let the group $\mathbb{C}^*$ act on $X$ trivially, and define its action on $\Nt$ by rescaling the fibers of $\op{pr}$, which are all isomorphic to the vector space $\n$, via the character $z\mapsto z^{-2}$. This action commutes with the action of $G$ on $\Nt$ and $X$. It is easy to check that, with respect to these actions, $\op{pr}$ is in fact $G\times \mathbb{C}^*$-equivariant. 

\paragraph{Relation to the quantum groups.}We will be interested in calculating some particular sheaf cohomology groups over $X$ that are used in the description of  
the center of the principal block of the small quantum group associated with the Lie algebra $\g$ according to the main theorem in 
\cite{BeLa}. Let us recall the result and the setting. 

 Let $R$ be a finite reduced root system in a 
$\C$-vector space $\h$ and  fix
a basis of simple roots $S =\{\al_i, i \in I\}$. Let $\check\al$ denote
the coroot corresponding to the root $\alpha \in R$. The Cartan matrix
is given by $a_{ij} = \langle \al_i, \check\al_j \rangle$, where
$\langle \cdot , \cdot \rangle$ is the canonical pairing
$\h^* \times \h \to \C$. Let $W$ be the Weyl group of $R$.
 There exists a unique
$W$-invariant scalar product in $\h$ such that $(\al, \al)=2$ for
any short root $\al \in R$. Set $d_i = \frac{1}{2} (\al_i, \al_i)
\in \{ 1,2,3 \}$ for each $i \in I$.
We denote by $\Y =\Z R$ the root lattice, and by
$\X = \{ \mu \in \h : \langle \mu, \check\al \rangle 
\in \Z\;\; \forall \; \al \in R \}$
the weight lattice corresponding to $R$. The coweight lattice is
$\check\Y = \op{Hom}(\Y, \Z) \in \h^*$. Let $R^+$ be the set of positive roots,
define the dominant weights by
$\X^+ = \{ \mu \in \X : \langle \mu, \check\al \rangle
\geq 0 \;\; \forall \; \al \in R^+ \}$
and set $\Y^+ = \Y \cap \X^+$.

Let $G$ be a complex connected semisimple group of adjoint type 
with the Lie algebra $\g$ corresponding to the root system $R$.

Let $\C(q)$ denote the field of rational functions in the variable $q$.
We denote by $U_q(\g) = U_q $ the Drinfeld-Jimbo quantized enveloping algebra
of $\g$. It is generated over $\C(q)$ by $E_i, F_i, i \in I$ and
$K^{\pm 1}_\mu, \mu \in \check\Y$ subject to well-known relations,
see e.g. \cite{Lus4}. We will write $K_i$ for $K_{d_i \check \al_i}$.
The algebra $U_q$ is a Hopf algebra over $\C(q)$.

Fix an odd positive integer $l$ which is greater than the Coxeter
number of the root system,
prime to the index of connection $|\X/\Y|$ and
prime to $3$ if $R$ has a component
of type $G_2$. Choose a primitive $l$-th root of unity $\xi \in \C$
and let $\A \subset \C(q)$ be the ring localized at $\xi$, and $\m$
the maximal ideal of $\A$. For any $n \in \N$ set
$[n]_d = \frac{q^{dn} - q^{-dn}}{q^d -q^{-d}}$ and
$[n]_d ! = \prod_{s=1}^n \frac{q^{ds} - q^{-ds}}{q^d -q^{-d}}$.

In $U_q$ consider the divided powers of the generators
$E_i^{(n)} = E_i^n /[n]_{d_i}!, F_i^{(n)} = F_i^n/[n]_{d_i}!,
i \in I, n \geq 1$,
and $\big[{ {K_\mu,m} \atop{n}} \big] $ as defined in \cite{Lus4}. The
Lusztig's integral form $\Ul_\A$ is defined as an $\A$-subalgebra of
$U_q(\g)$ generated by these elements. $\Ul_\A$ is a Hopf subalgebra of
$U_q$. The Lusztig quantum algebra
at a root of unity $\Ul$ is defined by specialization of $\Ul_\A$ at $\xi$:
 $\Ul = \Ul_\A / \m \Ul_\A$. It has a Hopf algebra structure over $\C$.

The {\it small quantum group} $u_q(\g)$ is the subalgebra
in $\Ul$ generated by the elements  $E_i$, $F_i$ and the Cartan elements $\frac{K_i -K_i}{q^{d_i} -q^{-d_i}}$, $i \in I$. Since we have assumed $l$ to be odd, $\ul$
is a Hopf algebra over $\C$.

Let $\op{Rep}(\ul)$ be the category of finite-dimensional
$\ul$-modules over $\C$. The finite-dimensional Hopf algebra $\ul$ decomposes as a left $\ul$-module
into a finite direct sum of finite-dimensional submodules.
Denote by $\ul_0$ the largest direct summand for which all its
simple subquotients belong to the principal block of the category
$\op{Rep}(\ul)$. Then $\ul_0$ is a two-sided ideal in $\ul$,
which will be called the principal block of $\ul$. Let $\zl$ denote the center of $\ul$. 
It decomposes into a direct
sum of ideals according to the block decomposition of $\ul$.
Set $\zl_0 = \zl \cap \ul_0$.

Recall that we have a $G$-equivariant isomorphism of vector bundles
$G \times_{B} \n \cong T^*X = \Nt$, and that the
multiplicative group acts on $\Nt$ by dilations on the fibers: an
element $z \in \C^*$ acts on $\n$ by multiplication by $z^{-2}$.
Consider the  coherent sheaf of poly-vector fields $\Lambda^\bullet
T\Nt$ on $\Nt$. The direct image of this sheaf to $X$ is in fact
bi-graded. The first grading is the natural grading
$\Lambda^\bullet T\Nt =\oplus_{j = 0}^{ \op{dim}(\Nt)} \Lambda^j
T\Nt$. The second grading comes from the induced action of $\C^*$
on $\Nt$. We will write $\mathrm{pr}_*(\Lambda^j T\Nt)^k$ for the $(j,k)$-th
component with respect to this bi-grading; this is a locally free
$G$-equivariant coherent sheaf on $X$. Notice that,
because of the definition of the $\C^*$-action,
$\mathrm{pr}_*(\Lambda^j T\Nt)^k=0$ for odd $k$.

Then we have the following result \cite{BeLa}: 

\begin{theorem} \label{Hoh}
There exists an isomorphism of algebras between
the total Hochschild cohomology of the principal block $\ul_0$ and
the total cohomology of $\Nt$ with coefficients
in $\Lambda^\bullet T\Nt$; here the algebra structure on the second space
comes from multiplication in the exterior algebra $\Lambda^\bullet T\Nt$.
The isomorphism is compatible with the grading as follows:
$$  \op{HH}^s (\ul_0) \cong \bigoplus_{i+j+k=s} \op{H}^i(\Nt, \Lambda^j T\Nt)^k. $$
In particular, 
$$ \zl_0 \cong \op{HH}^0 (\ul_0) \cong \bigoplus_{i+j+k=0} \op{H}^i(\Nt,\Lambda^j T\Nt)^k. $$
\end{theorem}

\paragraph{Classical results.}To compute the sheaf cohomology on the right-hand side of the isomorphism in Theorem \ref{Hoh}, 
one of the basic tools at our disposal is the Borel-Weil-Bott (BWB) Theorem~\cite{B} for cohomology of line bundles over $X$, which holds for the flag variety associated to any complex reductive Lie group. The equivariant line bundles on $X$ are classified by one-dimensional representations of $B$, which correspond bijectively to the characters of the maximal torus $H$.

\begin{theorem}[Borel-Weil-Bott] \label{BWB}
Let $\lambda \in \X$. If $\lambda +  \rho$ is  singular, then  $\mH^i(X, \Lx_\lambda) = 0$ for all $i$.  If $\lambda + \rho$ is regular, then there exists a unique nonvanishing cohomology group
$\mH^{i(\lambda)}(X, \Lx_\lambda)$, where $i(\lambda)$ is the length of the unique element $w$  of the Weyl group
such that $w(\lambda+\rho)$ is dominant. In this case, $\mH^{i(\lambda)}(X, \Lx_{\lambda})$ is an irreducible
representation of $G$ with highest weight $w \cdot \lambda = w(\lambda + \rho) - \rho$.
\end{theorem}

Another useful result is Borel's theorem on cohomology of $X$, which holds for a general complex reductive group  (see, e.g. \cite{Hil}).
\begin{theorem}[Borel] \label{thm-flag-cohomology} 
There is an algebra isomorphism between the cohomology  of the flag variety $X$  and  the algebra of coinvariants of the corresponding root system,
\[ 
\mH^\bullet(X) \cong \bigoplus_{i \geq 0} \mH^i(X, \Omega_X^i) \cong \frac{S^\bullet(\h)}{S^\bullet(\h)^W_+} .
\]
Here $S^\bullet(\h)$ is the symmetric algebra of a Cartan subalgebra $\h \subset \g$, and $S^\bullet(\h)^W_+$ denotes the augmentation ideal of the Weyl group invariants in $S^\bullet (\h)$.
\end{theorem}

We will also use the following theorem of Bott that relates the sheaf cohomology of vector bundles over $X = G/B$ with the relative cohomology of Lie algebras $(\lb, \, \h)$ with coefficients in a certain $\lb$-module.

\begin{theorem}[Bott~\cite{B}] \label{Bott_rel_lie}  Let $E$ be a holomorphic $B$-module, and $W$ a holomorphic $G$-module. 
Let ${\mathcal E}$ be the sheaf of local holomorphic 
sections of the equivariant vector bundle $G\times_B E$ on the flag variety $X$. Then there is an isomorphism of vector spaces
\[ 
\op{Hom}_G \left( W, \; \mH^{\bullet}(X, \, {\mathcal E}) \right) = \mH^{\bullet}(\lb, \, \h, \,  \op{Hom}(W,E) ) . 
\] 
\end{theorem} 

Here the left-hand side contains the sheaf cohomology of ${\mathcal E}$ over $X$, and the right-hand side represents the relative cohomology of the Lie algebras $(\lb, \, \h)$ with coefficients in the $\lb$-module $\op{Hom}(W,E)$. Suppose we know the structure of the $\lb$-module $E$
corresponding to the vector bundle over $X$ whose cohomology we need to compute. Using filtrations by line bundles and applying Theorem~\ref{BWB}, we can 
obtain a list of $G$-modules $W$ that can appear as direct summands of the $G$-module $\mH^{\bullet}(X, {\mathcal E})$. 
Then Theorem~\ref{Bott_rel_lie} allows us to compute the multiplicity of each $W$ in $\mH^{\bullet}(X, {\mathcal E})$ as the dimension of the relative cohomology of $(\lb, \, \h)$ with coefficients in $\op{Hom}(W,E)$. The last problem is easy to solve explicitly if the considered modules are relatively small, which will be the case for the sheaves that appear
 in Theorem \ref{Hoh} for $G = SL_3(\C)$.


Furthermore, we have 

\begin{lemma}(\cite[Corollary 5.2]{B}). \label{H0_H1_rel_lie}
Let $F$ be a finite-dimensional module over the Borel subalgebra $\lb$ of a complex semisimple Lie algebra $\g$, 
$\n = [\lb, \lb]$ the nilradical, and $\h$ be the Cartan subalgebra of $\lb$. Suppose $F$ is $\h$-diagonalizable. Then 
 the relative Lie cohomology of the pair $(\lb, \h)$ with values in $F$ is the $\h$-invariant submodule of the cohomology of the Lie algebra $\n$ with values in $F$: 
\[   \mH^\bullet(\lb, \h, F) \cong \mH^\bullet(\n, F)^\h .\]
In particular,    
 $ \mH^0(\lb, \h, F) \cong F^\lb$,  the $\lb$-invariant submodule of $F$. \hfill $\square$
\end{lemma} 




\paragraph{Lie algebra cohomology and the BGG resolution.} We can make Bott's method (Theorem \ref{Bott_rel_lie} and Lemma \ref{H0_H1_rel_lie}) more combinatorially approachable by using the Bernstein-Gelfand-Gelfand (BGG) 
resolution \cite{BGG,BGG2}.  Let $\lambda\in \X^+$ be a dominant integral weight for $\g$, and denote by $L_\lambda$ the corresponding simple $\g$-module with highest weight $\lambda$. 
By Theorem \ref{Bott_rel_lie} and Lemma \ref{H0_H1_rel_lie}, we need to compute the cohomology 
\[ 
 \op{H}^\bullet (\lb, \h, \op{Hom}(L_\lambda, E))  \cong  \op{H}^\bullet(\n,  \op{Hom}(L_\lambda, E))^{\h}  ,
\]
where $E$ is some finite-dimensional 
$\lb$-module. By definition we have 
\begin{equation}\label{eqn-n-cohomology}
\begin{array}{rcl}
 \op{H}^\bullet(\n,  \op{Hom}(L_\lambda, E))^\h &\cong & \op{Ext}_{\n}^\bullet(\C, \op{Hom}(L_\lambda, E))^\h 
\cong \op{Ext}_{\n}^\bullet(\C, E \otimes L_\lambda^*)^\h \\
& \cong &  \op{Ext}_{\n}^\bullet(L_\lambda \otimes E^* , \C)^\h .
\end{array}
\end{equation}

To compute the last term we need, by definition, an $\h$-graded projective resolution for the $U(\n)$-module $L_\lambda\otimes E^*$. Such a minimal resolution is provided by the BGG resolution for $L_\lambda$, which we briefly recall below.  

For the simple $\g$-module $L_\lambda$, the BGG theorem \cite{BGG, BGG2} provides a resolution by direct sums of Verma modules (see also \cite[Section 6.1]{Hum} for a concise exposition):
\begin{equation}\label{eqn-BGG-rsln}
 0 \lra M_{w_0 \cdot \lambda} \to \cdots \to \bigoplus_{l(w)=j} M_{w \cdot \lambda} \stackrel{d_j}{\lra} \bigoplus_{l(w)=j-1} M_{w \cdot \lambda} \to 
\cdots \to M_\lambda \stackrel{\varepsilon}{\longrightarrow} L_\lambda \lra 0 .
\end{equation}
Here $M_\mu$ denotes the Verma $\g$-module of highest weight $\mu\in \X$, $l:W\lra \N$ is the length function on $W$ and $w_0$, as usual, stands for the longest element in the group $W$. Recall that if we write $\lb^+=\h+\lu$, then $M_\mu=U(\g)\otimes_{U(\lb^+)}\C_{\mu}$, where $\lb^+$ acts on the one-dimensional space $\C_{\mu}$ via the $\h$-character $\mu$. By restriction to $\n$, we obtain an $\h$-graded free $U(\n)$-module resolution for $L_\lambda$. If $E^*$ is another $\h$-graded $U(\n)$-module, then the tensor product of the complex \eqref{eqn-BGG-rsln} with $E^*$ provides a desired resolution for $L_\lambda\otimes E^*$ because $U(\n)\otimes E^*$ is a free $U(\n)$-module\footnote{This is more generally true for any Hopf algebra $H$: the tensor product of a projective $H$-module with an arbitrary module remains projective.}. Denote the complex \eqref{eqn-BGG-rsln} without the $L_\lambda$ term by $M_\bullet$. Then the groups $\mathrm{Ext}_\n(L_\lambda\otimes E^*,\C)$ appearing in equation \eqref{eqn-n-cohomology} are identified with the cohomology groups of following complex the length $l(w_0)$: 
\begin{equation}\label{F_BGG}
\mathrm{Hom}_{\n}(M_\bullet\otimes E^*, \C)^{\h}\cong
\left(
\cdots \lra \bigoplus_{l(w) =j-1} E[w\cdot \lambda] \stackrel{d_j^*}{\longrightarrow} \bigoplus_{l(w) =j} E[w\cdot \lambda]\lra \cdots 
\right),
\end{equation}
 where $E[\mu]$ denotes the $\mu$-weight subspace of $E$. Here we have used the fact that
 \[
\mathrm{Hom}_{U(\n)}(M_\mu, E)^{\h}\cong (E_{-\mu})^{\h}\cong E[\mu],
 \]
 with $E_{-\mu}$ standing for $E$ with the $\h$-action shifted by the character $-\mu$. The maps $d_j^*$ are induced from the maps $d_j$ in the BGG resolution \eqref{eqn-BGG-rsln} by adjunction. We will give a more detailed description of these maps below.  Combining Theorem \ref{Bott_rel_lie} and this discussion, we obtain the following statement. 
 
\begin{proposition}  \label{rel_lie_BGG}
Let $E$ be a finite-dimensional $\lb$-module, and $L_\lambda$ be the finite-dimensional simple $\g$-module of dominant highest weight $\lambda$. Then the $\N$-graded multiplicity space
\[ 
\op{Hom}_G \left( L_\lambda, \; \mH^{\bullet}(X, G\times_B E) \right)  
\] 
is given by the cohomology of the complex 
\begin{equation*} 
0\lra E[\lambda]\to \cdots \to \bigoplus_{l(w) =j-1} E[w\cdot \lambda] \stackrel{d_j^*}{\longrightarrow} \bigoplus_{l(w) =j} E[w\cdot \lambda]\to \cdots \to E[w_0\cdot \lambda]\lra 0, 
\end{equation*}
where $w_0$ stands for the longest element of the Weyl group for $G$ and $j\in\{1,2,\dots, l(w_0)\}$. \hfill$\square$
\end{proposition}

Our next goal is to compute explicitly the boundary maps $d_j^*$. 

We fix a $\lambda\in \X^+$, and for $w \in W$,  take an embedding $M_{w \cdot \lambda} \to M_\lambda$, which also determines the embeddings 
$M_{w' \cdot \lambda} \to M_{w \cdot \lambda}$ whenever $w < w'$ in the Bruhat order on $W$. Since the maps between Verma modules are unique up to scalars (see, for instance, \cite[Theorem 4.2]{Hum}), this amounts to specifying a maximal $\h$-weight vector $u v_{w\cdot\lambda} \in M_{w\cdot \lambda}\cong U(\n) v_{w\cdot\lambda}$ of weight $w^\prime \cdot \lambda$, where $u \in U(\n)$ has weight $(w^\prime\cdot \lambda-w\cdot \lambda)$ and $v_{w\cdot \lambda}$ is a free $U(\n)$-module generator of weight $w\cdot\lambda$. Then the map from $M_{w' \cdot \lambda}$ to $ M_{w \cdot \lambda}$ is determined by a scalar coefficient $e(w,w')$ which is independent of $\lambda$. Write $w \to w'$ whenever $w' =tw$ for a $t \in W$ and $l(w') = l(w)+1$.  Then the maps between the Verma modules  appearing in the BGG resolution are given by the collections of scalars $e(w,w')$ such that $w \to w'$.  
The theorem in \cite{BGG2} states that for all pairs $w \to w'$ the scalars $e(w,w')$ can be chosen so that $e(w,w') \in \{ \pm 1 \}$ and the BGG resolution is an exact complex. For example, a possible choice for $\g = \lie{sl}_3$ is shown in the diagram 
\ref{sl3_BGG}. Each arrow $w \to w'$ is labeled with the value of the scalar $e(w,w')$. 

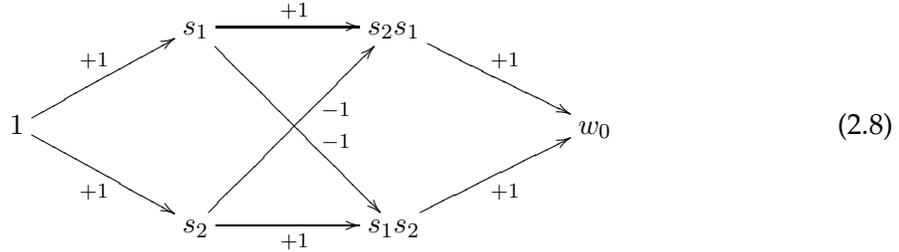
\begin{equation} \label{sl3_BGG}
\begin{gathered}
\xymatrix{
&& s_1\ar[rr]^{+1}\ar[ddrr]^{\hspace{0.1in}-1}&& s_2s_1 \ar[drr]^{+1} && \\
1 \ar[urr]^{+1}\ar[drr]_{+1} && && && w_0\\
&& s_2\ar[rr]_{+1}\ar[uurr]_{\hspace{0.1in}-1} && s_1s_2\ar[urr]_{+1} &
}
\end{gathered} 
\end{equation}


Now suppose that $w, w ' \in W$ are such that $w \to w'$, and $w' = s_\beta w$ for a reflection $s_\beta$ corresponding to a positive root $\beta$. Then the Verma module $M_{w \cdot \lambda} \cong U(\n) v_{w \cdot \lambda}$ with a fixed maximal weight vector $v_{w \cdot \lambda}$ contains a maximal vector $u_\beta v_{w \cdot \lambda}$ of weight $s_\beta w \cdot \lambda$ for a certain element $u_\beta \in U(\n)$, 
and the map $M(s_\beta w \cdot \lambda) \to M(w \cdot \lambda)$ in the BGG resolution 
can be realized as the right multiplication by 
$ e(s_\beta w,  w) u_\beta$ on $U(\n)$.
 
 In particular, if $s_i \in W$ is a simple reflection corresponding to the simple root $\alpha_i$, then the map $M_{s_i \cdot \lambda} \lra M_{\lambda}$ in the BGG resolution  can be realized as 
 $$
 (-)\cdot e(s_i,  1) f_i^{\langle  \lambda, \check{\alpha}_i \rangle+1} : U(\n) \lra U(\n), \quad x\mapsto e(s_i,  1)x\cdot f_i^{\langle  \lambda, \check{\alpha}_i \rangle+1}.
 $$
  
In general, the elements $u_\beta \in U(\n)$ are more complicated and should be determined in each particular 
case to make sure that BGG resolution is an exact complex. In all cases considered in this paper we will be able to compute  $u_\beta$ explicitly. For example, below is the complete diagram for $\g = \lie{sl}_3$ and $\lambda=0$. In this case, we identify all the Verma modules appearing in the complex \eqref{eqn-BGG-rsln} as certain left ideals in $U(\n)\cdot 1\cong M_0$. 
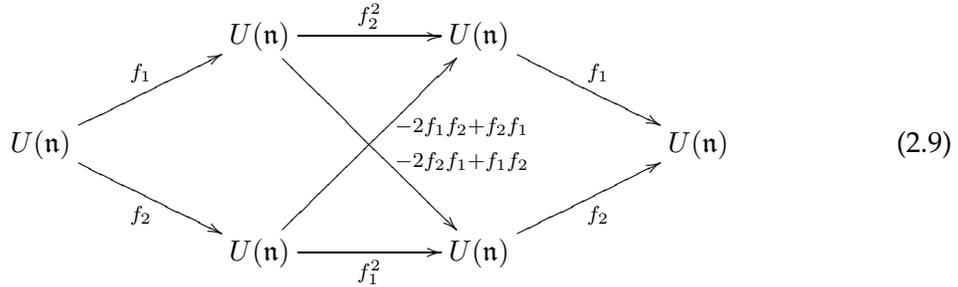
\begin{equation} \label{eqn-BGG-sl3-freemod}
\begin{gathered}
\xymatrix{
&& U(\n)\ar[rr]^{f_2^2}\ar[ddrr]^{\hspace{0.1in} {\small  -2f_1f_2+f_2f_1}}&& U(\n) \ar[drr]^{f_1} && \\
U(\n) \ar[urr]^{f_1}\ar[drr]_{f_2} && && && U(\n)\\
&& U(\n)\ar[rr]_{f_1^2}\ar[uurr]_{\hspace{0.1in}{\small -2f_2f_1+f_1f_2}} && U(\n)\ar[urr]_{f_2} &
}
\end{gathered}
\end{equation}
The elements on the arrows in the diagram indicate right multiplications by the elements on the free module $U(\n)$. (A similar computation for $\g = \mathfrak{sl}_4$ is given in Section \ref{sec-further-evidence}). 
It is worth pointing out that \eqref{eqn-BGG-sl3-freemod} is indeed a complex due to the $\mathfrak{sl}_3$-Serre relations
\[
f_1^2f_2+f_2f_1^2=2f_1f_2f_1, \quad \quad 
f_2^2f_1+f_1f_2^2=2f_2f_1f_2.
\]
Now, applying $\mathrm{Hom}_{\n}(M_\bullet\otimes E^*, \C)^\h$, we obtain the following total complex
\begin{equation} \label{eqn-BGG-sl3}
\begin{gathered}
\xymatrix{
&& E[s_1\cdot 0]\ar[rr]^-{f_2^2}\ar[ddrr]^-{\hspace{0.25in} {\small -2f_2f_1+f_1f_2}}&& E[s_2s_1\cdot 0] \ar[drr]^-{f_1} && \\
E[0] \ar[urr]^{f_1}\ar[drr]_{f_2} && && && E[w_0\cdot 0]\\
&& E[s_2\cdot 0]\ar[rr]_-{f_1^2}\ar[uurr]_-{\hspace{0.25in} {\small -2f_1f_2+f_2f_1}} && E[s_1s_2\cdot 0]\ar[urr]_-{f_2} &
} 
\end{gathered}\ .
\end{equation}

More generally, it follows from this discussion that the differential maps in \eqref{F_BGG} are given by the left action of certain elements $U(\n)$ multiplied by $\pm 1$.  
Suppose that $w \in W$ is such that   $l(w) =j-1$. Then  
\begin{equation}\label{eqn-boundary-map}
\left. d_j^*\right|_{E[w\cdot\lambda]} = \bigoplus_{\{s_\beta : w \to s_\beta w\}} e(s_\beta w, w) 
u_\beta : E[w \cdot \lambda ] \lra \bigoplus_{\{s_\beta : w \to s_\beta w\}} E[s_\beta w \cdot \lambda]  .  
\end{equation}

The complex \eqref{F_BGG}, together with the explicit differential maps \eqref{eqn-boundary-map}, will be used to compute the multiplicity space for the $G$-isotypical components of $\mH^\bullet(X,G\times_BE)$ once we know the structure of the $\lb$-modules $E$. Our next goal is to determine the structure of $E$ that will appear in the computation of the center of the principal block of $u_q(\g)$.

\subsection{The equivariant structure of the tangent bundle}\label{sec-equ-structure-tan}
 This and the next subsecions are devoted to understanding the $\lb$-module structure of the module corresponding to the pushdown of the tangent bundle and its exterior powers from $\Nt$ to $X$.

Since $\op{pr}:\Nt\lra X$ is a $G\times \C^*$-equivariant projection, the pushforward of the tangent bundle $\op{pr}_*(T\Nt)$ onto $X$ is a $G\times \C^*$-equivariant vector bundle.  Our goal is to represent $\op{pr}_*(T\Nt)$ as a bundle of the form $G\times_B V$ for some infinite-dimensional $B$-representation.

We start by recalling the following result.

\begin{lemma}\label{lemma-affineness-pr}
The morphism $\op{pr}:\Nt\lra X$ is affine. In particular, it induces an equivalence between the abelian categories of coherent sheaves on $\Nt$ and quasi-coherent sheaves on $X$ that are finitely generated over the sheaf of algebras $\op{pr}_*(\mathcal{O}_{\Nt})$.
\end{lemma}
\begin{proof}
The first claim is clear, since the fibers of $\op{pr}$ are all affine spaces that are isomorphic to $\n$. The second claim is a general property of affine maps. See\cite[Exercise II.5.17]{Ha}.
\end{proof}

As $\lu$ is the space of $\mathbb{C}$-linear functions on $\n$,  we identify the infinite-rank, locally free sheaf of algebras $\op{pr}_*(\mathcal{O}_{\Nt})$ as the associated sheaf of an infinite-rank $G\times \C^*$-equivariant bundle:
\begin{equation}\label{eqn-pushforward-O}
\op{pr}_*(\mathcal{O}_{\Nt})\cong G\times_B S^\bullet(\lu).
\end{equation} 
Here $S^\bullet:=\oplus_{k\in \N}S^k$ stands for the direct sum of all symmetric powers of a vector space. With respect to the $\C^*$-action, the degree of a homogeneous local section in $S^k(\lu)$ is equal to $2k$.

The tangent bundle of $X$ (see \eqref{eqn-tan-cotan-as-equiv-bdl}) fits naturally into the following short exact sequence of vector bundles on $X$:
\begin{equation}\label{eqn-ses-one-on-X}
G\times_B\left(
\xymatrix{
0 \ar[r] & \lb \ar[r] &  \g \ar[r] & \lu\ar[r] & 0
}
\right).
\end{equation}
The middle term is isomorphic, although not $G\times \C^*$-equivariantly, to the trivial vector bundle $X\times \g$ (Example~\ref{eg-equiv-bundle}).
Pulling the sequence \eqref{eqn-ses-one-on-X} back to $\Nt$ along $\pr$, we obtain a short exact sequence
\begin{equation}\label{eqn-ses-one-on-N}
\xymatrix{
0 \ar[r] & \pr^*(G\times_B \lb) \ar[r] &   \g \times \Nt\ar[r] &\pr^*(G\times_B \lu)\ar[r] & 0.
}
\end{equation}

The tangent bundle $T\Nt$ has a natural subbundle that consists of tangent vectors along the fiber. Since the fibers are linear spaces, this subbundle is isomorphic to the pullback of $\Nt$ itself along $\pr$. The quotient bundle of $T\Nt$ modulo this subbundle are naturally identified with the pullback of the tangent bundle of $X$. Therefore, we have another short exact sequence of vector bundles
\begin{equation}\label{eqn-ses-two-on-N}
\xymatrix{0\ar[r] & \pr^*(G\times_B \n)\ar[r] & T\Nt \ar[r] & \pr^*(G\times_B \lu) \ar[r] & 0.}
\end{equation}

The $G$-actions $G\times \Nt\lra \Nt$ and $G\times X\lra X$ give rise to infinitesimal actions that are the vertical arrows of the commutative diagram
\begin{equation}\label{eqn-inf-g-action-diag1}
\begin{gathered}
\xymatrix{
\g \times \Nt \ar[d] \ar[rr]^-{\mathrm{Id_\g}\times \pr} && \g \times X \ar[d]\\
 T\Nt \ar[rr]^-{d(\pr)} && TX
}
\end{gathered} 
\ .
\end{equation}
Here $d(\pr)$ stands for the total derivative of the projection map $\pr:\Nt\lra X$. Pulling back the rightmost vertical arrow to $\Nt$, we obtain a commutative diagram
\begin{equation}\label{eqn-inf-g-action-diag2}
\begin{gathered}
\xymatrix{
\g \times \Nt \ar[d]_{\phi} \ar@{=}[r] & \pr^*(\g \times X)  \ar[d]\\
T\Nt \ar[r] & \pr^*(TX)
}
\end{gathered} 
\ .
\end{equation}
The infinitesimal action map from $\g\times \Nt$ to $T\Nt$ has been denoted by $\phi$. Applying $\phi$ to the middle terms of \eqref{eqn-ses-one-on-N} and \eqref{eqn-ses-two-on-N}, we obtain a map of short exact sequences of vector bundles on $\Nt$:

\begin{equation}\label{eqn-map-ses-on-N}
\begin{gathered}
\xymatrix{
0 \ar[r] & \pr^*(G\times_B \lb)\ar[d] \ar[r] &   \g \times \Nt\ar[r] \ar[d]^{\phi} &\pr^*(G\times_B \lu)\ar[r] \ar@{=}[d]& 0\\
0 \ar[r] & \pr^*(G\times_B \n)\ar[r] & T\Nt \ar[r] & \pr^*(G\times_B \lu) \ar[r] & 0
}
\end{gathered}\ .
\end{equation}

The left most vertical map is induced from $\phi$. Let us look at it more closely. Since $B$ fixes the identity coset $eB\in X$ and $\pr$ is $G$-equivariant, $B$ acts naturally on the fiber over the coset, which is identified with $\n$, via the adjoint representation $\mathrm{Ad}_\n$. It follows that, upon differentiation, we obtain a linear map
\begin{equation}\label{eqn-ad}
\mathrm{ad}_\n: \lb\lra \mathrm{End}_\C(\n).
\end{equation}
In other words, the vector fields coming from the infinitesimal action of $\lb$ on $\Nt$ are ``vertical'' along the fiber over $eB$. Since the $\lb$-action is linear, we may use the $G$-invariant bilinear form fixed earlier to identify the map \eqref{eqn-ad} as
\begin{equation}\label{eqn-ad-b}
\mathrm{ad}_\n: \lb\lra  \lu\otimes \n,
\end{equation}
where, in the term $\lu\otimes \n$, $\lu$ is regarded as the space of linear functions on $\n$, and $\n$ is considered  to be the tangent space of the fiber. The map $\mathrm{ad}$ is $\C^*$-invariant and has degree zero because the action of  $\g$, and thus $\lb \subset \g$, commutes with the $\C^*$-action. This is reflected in the grading: the vector fields arising from $\lb$ (or $\g$) have degree zero, and the linear functions in $\lu$ have degree two, while the tangent vectors in $\n$ have degree minus two.

Now we can push forward the diagram \eqref{eqn-map-ses-on-N} onto $X$. By Lemma \ref{lemma-affineness-pr}, the short-exactness of the horizontal rows will be preserved. Via the projection formula (see, for instance, \cite[Exercise III.8.3]{Ha}) we obtain a map of short exact sequences of free $\pr_*(\mathcal{O}_{\Nt})$-modules on $X$:  
\begin{equation}\label{eqn-map-ses-on-X}
\begin{gathered}
\xymatrix{
0 \ar[r] & G\times_B (S^\bullet(\lu)\otimes\lb)\ar[d]^{\mathrm{ad}} \ar[r]^{\iota} &    G\times_B (S^\bullet(\lu)\otimes\g) \ar[r] \ar[d]^{\phi} & G\times_B (S^\bullet(\lu)\otimes\lu)\ar[r] \ar@{=}[d]& 0\\
0 \ar[r] & G\times_B (S^\bullet(\lu)\otimes\n)\ar[r] & \pr_*(T\Nt) \ar[r] & G\times_B (S^\bullet(\lu)\otimes\lu) \ar[r] & 0
}
\end{gathered}\ .
\end{equation}
In the diagram, $\iota$ is the natural inclusion map; the map $\ad$, induced from pushing forward $\ad_\n$, is given by the composition
\begin{equation}\label{eqn-ad-map}
\ad: S^\bullet(\lu) \otimes \lb \xrightarrow{\mathrm{Id}\otimes \ad_\n} S^{\bullet}(\lu) \otimes \lu \otimes \n \stackrel{\alpha}{\lra} S^{\bullet+1}(\lu)\otimes \n.
\end{equation}
Here $\alpha$ is the multiplication of polynomials by linear functions, or, in other words, it is just the symmetrization map.

Denote the direct sum of the maps $\iota$ and $\ad$ by $\Delta:$
\begin{equation}\label{eqn-diagonal-map}
\begin{array}{ccc}
\Delta: S^\bullet(\lu) \otimes \lb & \lra & S^\bullet(\lu) \otimes \g\oplus S^\bullet(\lu) \otimes \n, \\
 x & \mapsto & (\iota(x) , ~\ad(x)).
\end{array}
\end{equation}

The commutative square on the left of \eqref{eqn-map-ses-on-X} must be a pushout diagram, since the rightmost vertical map is an equality. It follows that as an equivariant bundle, $\pr_*(T\Nt)$ is of the form $G\times_B{V_1}$, where $V_1$ is the $B$-module
\begin{equation}\label{eqn-V1}
V_1:= \dfrac{S^\bullet(\lu) \otimes \g \oplus S^\bullet(\lu)\otimes \n }{\Delta(S^\bullet(\lu)\otimes \lb)} \ .
\end{equation}

Summarizing the discussion, we obtain the following statement. 

\begin{theorem}\label{thm-equ-structure-TN}
As a $G\times \C^*$-equivariant bundle on $X$, the pushforward tangent bundle $\pr_*(T\Nt)$ is isomorphic to
\begin{equation*}
\pr_*(T\Nt)\cong G\times_B V_1. 
\end{equation*}
\hfill$\square$
\end{theorem}

With respect to the $\C^*$-action, the module $V_1$ is $2 \Z$-graded. Since it is a free $S^\bullet(\lu)$-module, the grading is determined on the module generators. Elements in $\n$ have degree $-2$. The generators coming from $\g$ (and its subspace $\lb$) are homogeneous of degree zero, as the group action of $G$ and $B \subset G$ on $\Nt$ commutes with $\C^*$.  For example, the degree $-2$ part of the bundle $\pr_*(T\Nt)$ equals $G\times_B \n$, and the degree $0$ part is given by
\begin{equation}\label{eqn-tan-bundle-degree-zero-part}
G\times_B\left(\dfrac{\g\oplus \lu\otimes \n}{ \Delta(\lb)}\right),
\end{equation}
where $\Delta(x)=(x,\ad_\n(x))$ for any $x\in \lb$. This leads to 
  the following interesting observation, which we record here since it does not appear to be mentioned in the standard references.

\begin{corollary}\label{cor-deformation}
Assume $\g$ is a simple Lie algebra. Let $\mH^0(\Nt, T\Nt)^0$ denote the space of homogeneous vector fields  of degree zero with respect to the given $\C^*$-action on $\Nt$. Then $\mH^0(\Nt, T\Nt)^0$ is spanned by the vector fields coming from the infinitesimal $G$-action, and the Euler field generated by the infinitesimal $\C^*$-action along the fibers:
\[
\mH^0(\Nt,T\Nt)^0\cong \g\oplus \C.
\]
Furthermore, if $\g=\mathfrak{sl}_m$, then the degree-zero infinitesimal deformation space is isomorphic to the $G$-module\footnote{The result fails for other Lie types. Using similar methods and the equation \cite[(5.23)]{ViXue}, one can compute that for type $B_2$ the deformation space equals $\h\otimes \g \oplus L_{\alpha_1+\alpha_2}$.}
\[
\mH^1(\Nt, T\Nt)^0\cong \h \otimes \g,
\]
where $\h$ is equipped with the trivial $G$-action. 
\end{corollary}
\begin{proof}
Evidently, the vector fields generated by the (infinitesimal) $G\times \C^*$-action on $\Nt$ have degree zero, since this group commutes with the $\C^*$-action. Conversely, by equation \eqref{eqn-tan-bundle-degree-zero-part}, we have a short exact sequence of vector bundles on $X$:
\[
0 \lra G\times_B \lb \stackrel{\Delta}{\lra} G\times_B(\g\oplus \lu\otimes \n)\lra \mathrm{pr}_*(T\Nt)^0\lra 0.
\]
Taking the long exact sequence in cohomology, the result is reduced to the following two claims:
\begin{itemize}
\item[(1)] $\mH^i(X,G\times_B \lb)= 0$ for all $i\in \N$,
\item[(2)] $\mH^0(X,G\times_B(\lu\otimes \n))=\C$.
\end{itemize} 
Claim (1) appears in \cite[Section 5.1]{ViXue}. It follows from taking cohomology of the sequence \eqref{eqn-ses-one-on-X}, combined with the fact that 
\[
\mH^0(X,G\times_B\lu)\cong \mH^0(X,TX)\cong \g \cong \mH^0(X,G\times_B\g),\]
\[
\mH^i(X,G\times_B\lu)\cong \mH^i(X,TX)\cong 0 \cong \mH^i(X,G\times_B\g) \ \ (i\geq 1).
\]
Claim (2) is a special case of the general fact that a (partial) flag variety associated to a simple Lie algebra $\g$ is \emph{stable}, and stable manifolds have the property 
$$\mH^0(X,{\mathcal End}_{{\mathcal{O}_X}}(TX))\cong \C.$$
We formulate it as a separate Lemma \ref{lemma-one-diml-section-End}.

To show the last part of the Corollary, by claim (1), we are again reduced to compute
\[
\mH^1(T\Nt)^0 \cong \mH^1(X, G\times_B(\g\oplus \lu\otimes \n ))\cong \mH^1(X,G \times_B(\lu \otimes \n)),
\]
where the second isomorphism holds since $\mH^1(X,G\times_B\g)\cong 0$. When $\g=\mathfrak{sl}_m$, one has the following useful formula \cite[(5.7)]{ViXue}:
\[
\mH^k(X,\lb\otimes \lu)=
\left\{
\begin{array}{cc}
\C & k=0,\\
0 & k\neq 0.
\end{array}
\right.
\]
Taking cohomology of the short exact sequence of vector bundles on $X$: 
\[
0\lra G\times_B (\lu\otimes \n) \lra G\times_B(\lu\otimes \lb) \lra G\times_B (\lu \otimes \h)\lra 0, 
\]
and applying the above formula and claim (2), we obtain an isomorphism of $G$-representations
\[
\g\otimes \h \cong \mH^0(X,G\times_B(\lu\otimes \h)) \cong  \mH^1(X,G\times_B(\lu\otimes \n)).
\]
Here we have used that $G\times_B \h$ is a trivial bundle, so that 
$$\mH^0(X, G\times_B (\lu \otimes \h))\cong \mH^0(X,G\times_B \lu)\otimes \h \cong \g \otimes \h.$$ 
The result follows.
\end{proof}

\begin{lemma}\label{lemma-one-diml-section-End}
Let $\g$ be a simple Lie algebra, and $X$ its associated flag variety. Then the space of global sections in ${\mathcal End}_{{\mathcal{O}_X}}(TX)$ is one dimensional:
$$\mH^0(X,{\mathcal End}_{{\mathcal{O}_X}}(TX))\cong \C\cdot\mathrm{Id}_{TX}.$$
\end{lemma}
\begin{proof}
The result follows from a more general fact that stable vector bundles on a compact K\"{a}hler manifold have simple endomorphism algebras (see, for instance, Proposition 4.6.2 and Corollary 5.7.14 of \cite{Kobayashi}.
\end{proof}

\subsection{Extension to the exterior product bundles} \label{sec-ext-to-exterior}
Our goal in this subsection is to understand the equivariant structure of $\pr_*(\wedge^k T\Nt)$. As a first step, we consider 
the case of tensor products.

\begin{lemma}\label{lemma-equ-str-tensor} Let $\widetilde{W_i}$, $i=1,\dots, k$, be $G\times \C^*$-equivariant vector bundles on $\Nt$, and suppose that $W_i$, $i=1,\dots, k$, are graded free $S^\bullet(\lu)$-modules with compatible $B$-actions so that
\[
\pr_*(\widetilde{W_i})\cong G\times_BW_i \quad\quad (i=1,\dots,k).
\]
Then there is an isomorphism of $G\times \C^*$-equivariant vector bundles on $X$:
\begin{equation*}
\pr_*\left(\widetilde{W_1}\otimes_{\mathcal{O}_{\Nt}}\widetilde{W_2}\otimes_{\mathcal{O}_{\Nt}}\cdots\otimes_{\mathcal{O}_{\Nt}}\widetilde{W_k}\right)\cong G\times_B\left( W_1\otimes_{S^\bullet(\lu)}W_2\otimes_{S^\bullet(\lu)}\cdots \otimes_{S^\bullet(\lu)} W_k\right).
\end{equation*}
\end{lemma}
\begin{proof}This follows from the exactness of $\pr_*$ (Lemma~\ref{lemma-affineness-pr}) and the projection formula \cite[Exercise III.8.3]{Ha}.
\end{proof}

In what follows we will take all $\widetilde{W}_i$ to be $T\Nt$, so that the $W_i$s in Lemma \ref{lemma-equ-str-tensor} are equal to $V_1$ of \eqref{eqn-V1}. In this case, the symmetric group $S_n$ acts on both sides of the equation in Lemma \ref{lemma-equ-str-tensor}. Tensoring this equation with the sign character of the symmetric group, we get the following statement. 

\begin{corollary}\label{cor-exterior-product}
The pushforward of the exterior product bundle $\wedge^k T\Nt$ onto $X$ is a $G\times \C^*$-equivariant bundle of the form
\[
\pr_*(\wedge^k T\Nt)\cong G\times_B\left(\wedge^k_{S^\bullet(\lu)} V_1\right),
\] 
where $\wedge_{S^\bullet(\lu)}^k V_1$ stands for the exterior product of the graded free module $V_1$ over the polynomial algebra $S^\bullet(\lu)$. \hfill$\square$
\end{corollary}

Let $V_k$ denote the obtained graded $S^\bullet(\lu)$-module: 
\begin{equation}\label{eqn-Vn}
V_k:=\wedge^k_{S^\bullet(\lu)}V_1.
\end{equation}
It carries a natural $B\times \C^*$-action induced from the $B \times \C^*$-action on $V_1$.
We will also write the total exterior algebra of $V_1$ as
\begin{equation}\label{eqn-V-bullet}
V_\star:=\bigoplus_{k=0}^{2n}\wedge^k_{S^\bullet(\lu)}V_1,
\end{equation}
where $n=\mathrm{rank}(V_1)/2=\mathrm{dim}(X)$. The grading on $V_\star$ comes from the $\C^*$-action.  To write it out more explicitly, consider the natural surjective map 
\begin{equation}\label{eqn-surj-onto-Vn}
\wedge^k_{S^{\bullet}(\lu)}\left(S^{\bullet}(\lu)\otimes \g\oplus S^{\bullet}(\lu)\otimes \n\right)\lra V_k,
\end{equation}
which is obtained by taking exterior powers of the natural projection map
\[
S^{\bullet}(\lu)\otimes \g\oplus S^{\bullet}(\lu)\otimes \n\lra V_1.
\]
Summing over all $k\in \N$, we obtain a map of total exterior algebras 
\begin{equation}\label{eqn-surj-onto-V-star}
\wedge^\star_{S^{\bullet}(\lu)}\left(S^{\bullet}(\lu)\otimes \g\oplus S^{\bullet}(\lu)\otimes \n\right)\lra V_\star,
\end{equation}

Since $\Delta(S^\bullet(\lu)\otimes \lb)$ sits inside $S^{\bullet}(\lu)\otimes \g\oplus S^{\bullet}(\lu)\otimes \n$ as a direct $S^\bullet(\lu)$-module summand, the kernel of \eqref{eqn-surj-onto-V-star} equals the ideal generated by the submodule $\Delta(S^\bullet(\lu)\otimes \lb)$ on the left-hand side. It follows that the module $V_\star$ can be identified with
\begin{equation}\label{eqn-structure-V-star}
V_\star\cong \dfrac{\wedge_{S^{\bullet}(\lu)}^\star(S^{\bullet}(\lu)\otimes \g\oplus S^{\bullet}(\lu)\otimes \n)}{\Delta(S^\bullet(\lu)\otimes \lb)\wedge_{S^{\bullet}(\lu)} \left(\wedge_{S^{\bullet}(\lu)}^{\star-1}(S^{\bullet}(\lu)\otimes \g\oplus S^{\bullet}(\lu)\otimes \n)\right)} \ .
\end{equation}

For instance,  the term $V_2$ is isomorphic to
\begin{equation}\label{eqn-V-2}
V_2\cong
\dfrac{S^{\bullet}(\lu)\otimes \g\wedge \g \oplus S^{\bullet}(\lu)\otimes \g\otimes \n \oplus S^{\bullet}(\lu)\otimes \n\wedge \n}
{\Delta(S^{\bullet}(\lu)\otimes \lb)\wedge (S^{\bullet}(\lu)\otimes\g\oplus S^{\bullet}(\lu)\otimes \n)}.
\end{equation}
We can read off the homogeneous terms easily from this expression, using that $\mathrm{deg}(\lb)=\mathrm{deg}(\g)=0$, $\mathrm{deg}(\n)=-2$ and $\mathrm{deg}(\lu)=2$. Write 
\begin{equation}\label{eqn-homogeneous-term-V-star}
V_\star= \bigoplus_{k\in \Z}V_\star^k
\end{equation}
for the homogeneous decomposition of $V_\star$ with respect to the $\C^*$-grading.Then the lowest degree term of $V_2$ is just $V_2^{-4}=\n\wedge \n$ of degree $-4$. The degree $-2$ term equals
\begin{equation}\label{eqn-V2-minus-2}
V_2^{-2}=\dfrac{\g\otimes \n \oplus \lu \otimes (\n\wedge \n)}{\Delta_2(\lb\otimes \n)}.
\end{equation}
Here, $\Delta_2$ is the map induced from $\Delta=(\iota,\ad_\n):\lb\lra \g\oplus \lu\otimes\n$.
\begin{equation}\label{eqn-Delta2-map}
\Delta_2: \lb \otimes \n \xrightarrow{\Delta \otimes \mathrm{Id}_{\n}} \g\otimes \n \oplus \lu \otimes \n\otimes \n \stackrel{\beta}{\lra} \g\otimes \n \oplus \lu \otimes (\n\wedge \n),
\end{equation}
with $\beta$ being the anti-symmetrization map $\n\otimes \n \lra \n\wedge \n$.

Next, for $V_3$, the two lowest degree terms are given by
\begin{equation}\label{eqn-V3-minus-6}
V_3^{-6}=\n \wedge \n \wedge \n.
\end{equation}
\begin{equation}\label{eqn-V3-minus-4}
V_3^{-4}=\dfrac{\g\otimes (\n \wedge \n)\oplus \lu \otimes (\n\wedge \n\wedge \n)}{\Delta_3(\lb\otimes (\n\wedge \n))},
\end{equation}
where $\Delta_3$ is the composition map of $\Delta \otimes \mathrm{Id}_{\n\wedge \n}$ by the anti-symmetrization map $\beta$:
\begin{equation}\label{eqn-Delta-3}
\lb \otimes (\n\wedge \n) \xrightarrow{\Delta \otimes \mathrm{Id}_{\n\wedge\n}} \g\otimes (\n\wedge \n) \oplus \lu \otimes \n\otimes (\n\wedge\n) \stackrel{\beta}{\lra} \g\otimes (\n\wedge \n) \oplus \lu \otimes (\n\wedge \n \wedge \n).
\end{equation}

\paragraph{A duality.} In the last part of this section, we will discuss some basic duality results on the exterior product of the tangent bundle $T\Nt$. This will be applied later to cut down the amount of cohomological computations ``almost'' by half.

Recall that the Springer variety $\Nt$ is holomorphic symplectic, i.e., it is equipped with a non-degenerate, anti-symmetric $\Ox_{\Nt}$-linear pairing $\omega$:
\begin{equation}\label{eqn-symp-form}
\omega: T\Nt \otimes_{\mathcal{O}_{\Nt}} T\Nt \lra \mathcal{O}_{\Nt}.
\end{equation}
The symplectic form $\omega$ induces an isomorphism of bundles by contraction
\begin{equation}\label{eqn-tang-iso-cotang}
\iota_\omega : T\Nt \stackrel{\cong }{\lra} T^*\Nt.
\end{equation}
Furthermore, since the top exterior power of $T^*\Nt$ is trivial, we have, for each $k\in \{0,\dots, n\}$, a non-degenerate pairing
\begin{equation}\label{eqn-wedge-pairing}
\wedge^k T\Nt  \otimes_{\mathcal{O}_{\Nt}} \wedge^{2n-k}T\Nt\lra \wedge^{2n}T\Nt \cong \Ox_\Nt,
\end{equation}
given by fiberwise exterior product, where $n=\mathrm{dim}(\n)=\mathrm{dim}(X)$. It follows by combining 
\eqref{eqn-wedge-pairing} and \eqref{eqn-tang-iso-cotang} that
\begin{equation}\label{eqn-duality-Nt}
\wedge^k T\Nt  \stackrel{\textrm{\eqref{eqn-wedge-pairing}}}{\cong}  \mathcal{H}om_{\Ox_{\Nt}}(\wedge^{2n-k} T\Nt, \Ox_{\Nt}) \cong \wedge^{2n-k} T^*\Nt \stackrel{\textrm{\eqref{eqn-tang-iso-cotang}}}{\cong} \wedge^{2n-k}T\Nt.
\end{equation}
Pushing \eqref{eqn-duality-Nt} forward to the flag variety and applying Corollary \ref{cor-exterior-product}, we obtain the following result identifying the corresponding $B$-modules.

\begin{lemma}\label{lemma-half-reduction}
Let $n=\mathrm{dim}(X)$. For each $k\in \{0,\dots, n\}$, there is an isomorphism of graded $B$-modules
\[
V_k\cong V_{2n-k}.
\]
More precisely, there is an isomorphism of $B$-modules
\[
V_k^{-2r}\cong V_{2n-k}^{-2(n+r-k)}
\] 
for any $r\in \Z$.
\end{lemma}
\begin{proof} Without the grading, the statement follows directly by the discussion before the Lemma, and thus we are reduced to match the gradings involved.

To do so, consider $T_x\Nt$ where $x$ is any point on $\Nt$ living over the identity coset $eB\in X$. Choose a Chevalley basis for $\g$
$$\{e_1,\dots, e_n, h_1,\dots, h_l, f_1,\dots, f_n|e_i\in \lu, f_i\in \n~(i=1,\dots, n),~h_k\in \h,~(k=1,\dots, l)\}$$
such that  $\{e_1,\dots, e_n, f_1,\dots, f_n\}$ form a dual basis under the symplectic pairing $\omega$: 
\[
\omega(e_i,f_j)=-\omega(f_j,e_i) = \delta_{i,j}.
\]
Then, by the $S^\bullet(\lu)$-module structure of $V_1$ (see \eqref{eqn-V1}), the elements in the collections 
$$
\{\overline{e}_i:=(e_i,0)\in (\g\oplus \lu\otimes \n)/\Delta(\lb)|i=1,\dots, n\}, \quad \quad \{f_i\in \n|i=1,\dots, n\}$$ 
form an $S^\bullet (\lu)$-module basis for $V_1$. 

For any unordered subset $I=\{i_1,\dots, i_k\}\subset \{1,\dots, n\}$, denote by $I^o$ the complementary set $\{1,\dots, n\}\backslash I$. We will write
$$e_I:=e_{i_1}\wedge \cdots \wedge e_{i_k}, \quad f_I:=f_{i_1}\wedge \cdots \wedge f_{i_k}.$$

Given any size-$k$ subsets $I,J\subset \{1,\dots, n\}$, $\omega$ extends to the pairing on $\wedge^k(T_x\Nt)$ by setting
\[
\omega(e_I,f_J)=\omega(e_{i_1},f_{j_1})\cdots\omega(e_{i_k},f_{j_k})=\delta_{I,J}.
\]
It induces, via pushing forward, a non-degenerate $\op{pr}_*(\mathcal{O}_\Nt)$-linear antisymmetric pairing
\[
\op{pr}_*(\wedge^k T\Nt)\otimes_{\op{pr}_*(\mathcal{O}_\Nt)}\op{pr}_*(\wedge^k T\Nt) \lra \op{pr}_*(\mathcal{O}_\Nt),
\]
which by abuse of notation is still denoted as $\omega$. Upon restriction to the fiber over $eB$, we have an $S^\bullet(\lu)$-linear anti-symmetric pairing on $V_k$ defined similarly by
\[
\omega(\overline{e}_I,f_J)=\omega(\overline{e}_{i_1},f_{j_1})\cdots\omega(\overline{e}_{i_k},f_{j_k})=\delta_{I,J}.
\]
The wedge pairing \eqref{eqn-wedge-pairing} also descends to $X$ to be an $S^\bullet(\lu)$-linear map
\[
V_k\otimes_{S^\bullet(\lu)} V_{2n-k}\stackrel{\wedge}{\lra} V_{2n}\cong S^\bullet(\lu) \overline{e}_1\wedge \cdots \wedge \overline{e}_n\wedge f_1\wedge \cdots \wedge f_n .
\]
We can then check that the pushfoward of isomorphism \eqref{eqn-duality-Nt} at the identity coset sends $S^\bullet(\lu)$-basis elements consecutively to
\[
\begin{array}{ccccc}
V_k & \stackrel{\cong}{\lra}   & \mathrm{Hom}_{S^\bullet(\lu)}(V_{2n-k},S^\bullet(\lu)) & \stackrel{\cong}{\lra} & V_{2n-k}\\
\overline{e}_I\wedge f_{J} & \mapsto & \pm \overline{e}_{I^o}\wedge f_{J^o} & \mapsto & \pm \overline{e}{_{J^o}}\wedge f_{I^o}\ ,
\end{array}
\]
where $J\subset\{1,\dots, n\}$ is a subset of size $r$ and $I$ is of size $k-r$. Since $\mathrm{deg}(f_i)=-2$ and $\mathrm{deg}(\overline{e}_i)=0$,
the degree $-2r$ element $\overline{e}_I\wedge f_{J}
\in V_k^{-2r}$ is sent to $\pm \overline{e}{_{J^o}}\wedge f_{I^o} \in V_{2n-k}^{-2(n-(k-r))} $. The result now follows.
\end{proof}

\begin{corollary}\label{cor-half-reduction}
Let $n=\mathrm{dim}(X)$. For each fixed $i, k \in \{0, \dots, n\}$, there is an isomorphism of Hochschild cohomology groups
\[
\mH^i(\Nt,\wedge^kT\Nt)^{-2r}\cong \mH^i(\Nt,\wedge^{2n-k} T\Nt)^{-2n+2k-2r}.
\]
\end{corollary}
\begin{proof}
Use the isomorphisms
\begin{align*}
\mH^i(\Nt,\wedge^kT\Nt)^{-2r} & \cong \mH^i(X, G\times_B V_k^{-2r})\cong \mH^i(X, G\times_B  V_{2n-k}^{-2(n-k+r)})\\
 & \cong \mH^i(\Nt,\wedge^{2n-k}T\Nt)^{-2n+2k-2r}.
\end{align*}
The claim follows.
\end{proof}

\paragraph{Summary.}To conclude, let us summarize the main results of this section. 

By Theorem \ref{Hoh} cited from \cite{BeLa}, the degree-zero Hochschild cohomology of the principal block of $u_q(\g)$ can be computed as the sheaf cohomology of the push-forward polyvector fields $\mathrm{pr}_*(\wedge^k T \Nt)^{-2r}$ ($k,r\in \N$) over the flag variety $X=G/B$, where the degree $-2r$  comes from the $\C^*$ action.  Now Corollary \ref{cor-exterior-product} reveals the $B$-structure of the bundles 
$$\mathrm{pr}_*(\wedge^k T \Nt)^{-2r}\cong G\times_B V_k^{-2r},$$
 thus allowing us to use Bott's Theorem \ref{Bott_rel_lie}  to compute the multiplicity of each $G$-isotypical component of the center via (relative) Lie algebra cohomology.  Further, Proposition \ref{rel_lie_BGG} uses the BGG theory to establish an equivalence between the required Lie algebra cohomology and the cohomology of a complex defined entirely in terms of the action of $U(\n)$ on the finite-dimensional $B$-module $V_k^{-2r}$.  Finally Corollary \ref{cor-half-reduction} uses various dualities to obtain symmetries between the bigraded components of the center and reduce the computations almost by half. 

These results will allow us to compute the center of the principal block of $u_q(\mathfrak{sl}_3)$ as a bigraded vector space, and to outline the method for computing the structure of the center of the principal block 
for $u_q(\mathfrak{sl}_4)$. The method should work in general  for any semisimple Lie algebra $\g$; however the maps in the complex 
(\ref{F_BGG}) in Proposition \ref{rel_lie_BGG} depend on $\g$ and on the dominant weight $\lambda$  (although the latter dependence is minor since one can simply tensor the resolution for $L_0$ with $L_\lambda$), and they need to be determined in each case separately.


\section{The center of the principal block of quantum \texorpdfstring{$\mathfrak{sl}_3$}{sl(3)}}\label{sec-sl3-center}

Our goal in this section is to compute the sheaf cohomology groups appearing on the right-hand side of Theorem \ref{Hoh} in the case of $\g = \mathfrak{sl}_3$. We will use the general techniques developed in the previous section, as well as some methods 
specific to the given case. 

\subsection{Notation} \label{sec-ntn}
Let $G$ be the complex simple group $SL_3( \C)$ and $B \subset G$ be the Borel subgroup which consists of invertible lower triangular matrices of determinant one. We will denote by $\g$ the Lie algebra of $G$, which consists of traceless $3\times 3$-matrices, and by $\mathfrak{b}$ the Lie algebra of $B$ consisting of traceless lower triangular $3\times 3$ matrices. Also set $\n:=[\mathfrak{b},\mathfrak{b}]$ to be the Lie algebra of the unipotent subgroup $N:=[B,B]\subset B$ (strictly lower triangular matrices), and $\h$ to be the Lie algebra of the diagonal torus subgroup $H$ of $B$. The group $B$ acts on $\lb$  by the adjoint action, and $\n$ is an invariant subspace. The induced $B$-action on $\h\cong \lb/\n$ is then trivial.   The dual representation $\n^*$ of $\n$ is isomorphic to $\lu$ once we fix a $B$-invariant pairing, which we will do below.

Let us choose a Chevalley basis for the $B$-modules. Set
\begin{equation}\label{eqn-Chevalley-basis-ef}
f_1=\left(
\begin{matrix}
0 & 0 & 0\\
0 & 0 & 0\\
0 & 1 & 0
\end{matrix}
\right),\quad
f_2=\left(
\begin{matrix}
0 & 0 & 0\\
1 & 0 & 0\\
0 & 0 & 0
\end{matrix}
\right),\quad
e_1=\left(
\begin{matrix}
0 & 0 & 0\\
0 & 0 & 1\\
0 & 0 & 0
\end{matrix}
\right),\quad
e_2=\left(
\begin{matrix}
0 & 1 & 0\\
0 & 0 & 0\\
0 & 0 & 0
\end{matrix}
\right),
\end{equation}
\begin{equation}\label{eqn-Chevalley-basis-h}
h_1=\left(
\begin{matrix}
0 & 0 & 0\\
0 & 1 & 0\\
0 & 0 & -1
\end{matrix}
\right),
\quad
h_2=\left(
\begin{matrix}
1 & 0 & 0\\
0 & -1 & 0\\
0 & 0 & 0
\end{matrix}
\right) .
\end{equation}
The $\lb$-modules $\n$, $\lb$ and $\g$ are then spanned by 
\[ 
\n =\C f_1\oplus \C f_2 \oplus \C f_3, \quad  \lb = \n \oplus \C h_1\oplus \C h_2,  \quad \g =\lb\oplus \C e_1 \oplus \C e_2 \oplus \C e_3,
\]
where $f_3=[f_1,f_2]$ and $e_3=[e_2,e_1]$.
The subspace $\lu = \C e_1 \oplus \C e_2 \oplus \C e_3 \subset \g$  
 can be identified with $\n^*$ using the nondegenerate pairing 
$\langle e, f \rangle = \frac{1}{6} \op{tr}_{\g}(\op{ad}(e) \, \op{ad}(f))$ 
for any $e \in \lu$, $f \in \n$. With this definition we have $\langle e_i, f_j \rangle = \delta_{i,j}$. Then $\lu$ is a $\lb$-module 
by the (co)adjoint action: $\op{ad}_\lu(y) : \lu \to \lu$ defined by the formula $\langle \op{ad}_\lu(y) (e), \,f  \rangle = \langle e, -\op{ad}_\n(y) (f) \rangle$ for any $e \in \lu$, $y, f \in \n$.  
The Lie algebra $\lb$ maps to $\lu \otimes \n \cong \op{End}(\n)$ as it acts on $\n$ by the adjoint action: 
\begin{equation}\label{eqn-ad-sl3}
\op{ad}_{\n} : \lb \to \lu \otimes \n , \quad  \op{ad}_{\n}(y) = \sum \op{ad}_{\n}(y)_1 \otimes \op{ad}_{\n}(y)_2,
\end{equation}
where the components of $\op{ad}_{\n}(y)$ are defined by the condition  
 $$\op{ad}_{\n}(y)(f) = \sum \langle \op{ad}_{\n}(y)_1 , f \rangle \op{ad}_{\n}(y)_2. $$ 
 See equation \eqref{eqn-ad-b} and the discussion there for the general case.

Then we compute the effect of the map $\ad_\n$ on the basis of $\lb$: 
\begin{equation} \label{phi} 
 \begin{array}{lll} 
\op{ad}_{\n}(h_1)  &  = &  -2 e_1 \otimes f_1 + e_2 \otimes f_2 - e_3 \otimes f_3,   \\ 
\op{ad}_{\n}(h_2)  &  =  &  e_1 \otimes f_1 - 2 e_2 \otimes f_2  - e_3 \otimes f_3, \\ 
\op{ad}_{\n}(f_1) & =     & e_2 \otimes f_3, \\  
\op{ad}_{\n}(f_2)  &  =    &  - e_1 \otimes f_3, \\ 
\op{ad}_{\n}(f_3)  & = &  0 .  \\
\end{array} 
\end{equation}  

We will continue to use $X$ to denote the flag variety $G/B$ in this case, and write the tangent and cotangent sheaves as $T_X$ and $\Omega_X$ respectively.

In particular, for $G = SL(3, \C)$ the Theorems \ref{BWB} and \ref{thm-flag-cohomology} specialize to the following result:
\[ 
\mH^i(X, \Omega_X^j) 
\cong  
\left\{ 
\begin{array}{ll}
L_0,  &   i=j=0, \;{\rm or}\; i=j=3, \\
L_0^{\oplus 2},  & i=j=1, \; {\rm or} \; i=j=2, \\
0   & {\rm otherwise.}
\end{array} \right. 
\]
Here and below, $L_\lambda$ denotes the irreducible $G$-module with highest weight $\lambda$. In particular, $L_0 \cong \C$.


\subsection{Sheaf cohomology of some vector bundles over the flag variety}\label{sec-some-sheaves}
In this subsection, we will use various methods to compute the cohomology of some vector bundles over the flag variety in the case of $\mathfrak{sl}_3$. 
In particular, we will be interested in the sheaf cohomology 
of tensor product vector bundles of the form
\[  
\mH^\bullet (X,  \Omega^r_{X} \otimes \wedge^s T_{X} ) . 
\]
for the values of $r$ and $s$ required for the computation of the center. These vector bundles are composition factors 
in the sheaves  $\pr_*(\wedge^{r+s}T\Nt)$ that appear in Theorem \ref{Hoh}. 

For the ease of notation, we will sometimes drop the subscripts $X$ decorating vector bundles in this part, where it is understood that $X$ is the flag variety for $SL_3(\C)$. For instance, we will simply write the above cohomology groups as $\mH^i(\Omega^r\otimes \wedge^s T)$.

\paragraph{Cohomology of $~~\Omega_X \otimes T_X$.}
Consider the equivariant bundle $G\times_B (\n \otimes \lu)$ that corresponds to the  vector bundle $\Omega \otimes T$ over $X$. To illustrate the method developed in Seciton 2, we will compute its cohomology using the relative Lie algebra cohomology and the BGG resolution as proposed in Theorem \ref{Bott_rel_lie} and Proposition \ref{rel_lie_BGG}. 

Then for any dominant weight $\lambda$ 
the multiplicity of $L_\lambda$ in $\mH^\bullet(\Omega \otimes T)$ is given by the dimension of the relative Lie algebra cohomology $\mH^\bullet(\lb, \h, \op{Hom}(L_\lambda, \n \otimes \lu))$ which can be computed as 
the cohomology of the complex (\ref{F_BGG}) with $E= \n \otimes \lu$. 

The only dominant weights $\lambda$ such that $w \cdot \lambda$ appears as a weight of $\n \otimes \lu$ for some $w \in W$ are $\lambda =0$ and $\lambda = \rho$. Let us fix a basis of Chevalley generators $\{ e_1, e_2, e_{3} \}$ of $\lu$ and a dual basis $\{ f_1, f_2, f_{3} \}$ of $\n$ as in equation \eqref{eqn-Chevalley-basis-ef}. Then the nontrivial weight subspaces of $\n \otimes \lu$ 
with weights in the shifted-action orbit $\{w \cdot 0|w\in S_3\}$ are spanned by
\begin{itemize}
\item for $w=1$, $ \{  f_1 \otimes e_1 , f_2 \otimes e_2, f_{3} \otimes e_{3} \}$,
\item for $w =s_1$, $ \{ f_{3} \otimes e_2 \}$,
\item for $w =s_2$, $\{ f_{3} \otimes e_1 \}$.  
 \end{itemize}
Then  for  $\lambda =0$ the complex (\ref{F_BGG}) has the form 
\[ 0 \lra (\n \otimes \lu) [0] \stackrel{(f_1 , f_2)}{\longrightarrow}(\n \otimes \lu)[s_1 \cdot 0] \oplus (\n \otimes \lu) [s_2 \cdot 0] \lra 0 .\] 
An easy computation shows that the map $(f_1 , f_2)$ is surjective and its kernel is 1-dimensional. 
Therefore we have $\op{dim} ( \op{Hom}_G (\op{H}^0 (\Omega \otimes T) , L_0 ) ) =1$ and  
$\op{dim} (\op{Hom}_G (\op{H}^1 (\Omega \otimes T) , L_0 ) ) =0$. 

Now if $\lambda = \rho$, the only nontrivial weight subspaces of weights $w \cdot \rho$ for $w \in W$ are spanned by 
$  \{ f_{1} \otimes e_2 \} $ for $w=s_1$ and $ \{ f_{2} \otimes e_1 \}$ for $w= s_2$.  
Then the complex (\ref{F_BGG}) has nonvanishing terms concentrated in only one homological degree, namely
\[ 
0 \lra (\n \otimes \lu)[s_1 \cdot \rho] \oplus (\n \otimes \lu)[s_2 \cdot \rho]  \lra 0 .
\] 
Therefore, we obtain that $\op{dim}( \op{Hom}_G (L_\rho,\op{H}^1 (\Omega \otimes T)   ) ) =2$. 
Finally we conclude that the only nontrivial cohomology groups of $\Omega \otimes T$ over $X$ are 
$\op{H}^0(\Omega \otimes T) \cong L_0$ and $\op{H}^1(\Omega \otimes T) \cong L_\rho^{\oplus 2}$. 
The same results hold for the cohomology of $\Omega^2 \otimes \wedge^2 T \cong \Omega \otimes T$.  

\paragraph{Cohomology of $~~\Omega^2_X \otimes T_X$.} 
Next we consider the vector bundle $\Omega^2 \otimes T\cong G\times_B(\wedge^2 \n \otimes \lu)$. 
In a similar vein as above, one is reduced to computing the relative Lie algebra cohomology 
$\mH^\bullet(\lb, \h, \op{Hom}(L_\lambda,  \wedge^2 \n \otimes \lu))$ as 
the cohomology of the complex (\ref{F_BGG}) with $E=  \wedge^2 \n \otimes \lu$ and various integral dominant weights $\lambda$. 

The only dominant weight $\lambda$ such that $w \cdot \lambda$ is a weight of $\wedge^2 \n \otimes \lu$ is $\lambda =0$. 
The nontrivial weight spaces of weights $w \cdot 0$ for $w \in W$ are spanned by
\begin{itemize}
\item for $w=1$, $f_1 \wedge f_2 \otimes e_{3}$, 
\item for $w=s_1$, $\{ f_{1} \wedge f_2 \otimes e_2, f_1 \wedge f_{3} \otimes e_{3} \}$ ,
\item for $w=s_2$, $\{ f_{1} \wedge f_2 \otimes e_1, f_2 \wedge f_{3} \otimes e_{3} \}$ .
\end{itemize}  
  Then the complex (\ref{F_BGG}) 
 for $E = \wedge^2 \n \otimes \lu$ and $\lambda =0$ becomes 
 \[ 
 0 \lra (\wedge^2\n \otimes \lu)[0] \stackrel{(f_1, f_2)}{\longrightarrow}(\wedge^2\n \otimes \lu)[s_1 \cdot 0] \oplus (\wedge^2\n \otimes \lu) [s_2 \cdot 0]  \lra 0 .
 \] 
 An easy computation shows that the map $(f_1 , f_2)$ is injective and its image is one dimensional. 
Therefore, we have 
$$\op{dim}( \op{Hom}_G ( L_0,\op{H}^0 (\Omega^2 \otimes T)  ) ) =0,\quad  
\op{dim}( \op{Hom}_G (L_0,\op{H}^1 (\Omega^3 \otimes T)   ) ) =3.$$ Finally, the only nontrivial cohomology of $\Omega^2 \otimes T$ is $\op{H}^1(\Omega^2 \otimes T) \cong L_0^{\oplus 3}$.

\paragraph{An alternative approach.} 
For comparison, we include  
a more traditional way to compute the cohomology of these two bundles by using the BWB Theorem \ref{BWB} and the geometry of this particular case. 

The two lowest degree cohomology groups of $\Omega \otimes T$ has been covered in the proof of Corollary \ref{cor-deformation}. However, here we will use the BWB Theorem to show that all other cohomology groups of this sheaf vanish and simultaneously compute the first two non-zero terms.
 
The vector bundle $T=TX$ has the following filtration by line bundles
\begin{equation} 
0 \lra  \Lx_{\alpha_1} \oplus \Lx_{\alpha_2} \lra T \lra \Lx_{\alpha_1 + \alpha_2} \lra 0 ,  \label{T_filt} 
\end{equation}
where $\alpha_1$ and $\alpha_2$ are the simple roots of the $A_2$ root system.
The dual vector bundle $\Omega=\Omega_X$ has the dual filtration:
\begin{equation} 
0 \lra  \Lx_{-\alpha_1 - \alpha_2} \lra \Omega \lra  \Lx_{-\alpha_1} \oplus \Lx_{-\alpha_2} \lra 0 . \label{Om_filt} 
\end{equation}
Then the vector bundle $\Omega \otimes T$ has the following components in the filtration:
\[ 
\Qx_1 :=  \Lx_{-\alpha_1} \oplus \Lx_{-\alpha_2}  \subset \Qx_2 :=  \Qx_1 \oplus \Lx_0^{\oplus 3} \oplus \Lx_{\alpha_1 - \alpha_2} \oplus
\Lx_{\alpha_2 - \alpha_1}  \subset \Qx_3 :=  \Qx_2 \oplus \Lx_{\alpha_1} \oplus \Lx_{\alpha_2} . 
\]

Let us apply Theorem \ref{BWB}. Taking $\alpha_1 + \rho = 2\alpha_1 + \alpha_2$ and $\alpha_2 + \rho = 2 \alpha_2 + \alpha_1$, we observe that all weights of $\Qx_3/\Qx_2$ are singular, and $\Qx_3$ and $\Qx_2$ have isomorphic cohomology groups. Also,
all cohomology of degree higher than
one in the above filtration vanishes. Then we have the long exact sequence of cohomology
\begin{equation}\label{eqn-les-Omega-Tensor-T}
  \mH^0(\Qx_1) \to \mH^0(\Omega \otimes T) \to \mH^0(\Qx_2/\Qx_1) \to \mH^1(\Qx_1) \to \mH^1(\Omega \otimes T) \to \mH^1(\Qx_2/\Qx_1) \to 0. 
\end{equation}

Using Theorem \ref{BWB} for the given line bundles, we obtain $\mH^0(\Qx_1) =0$, $\mH^1(\Qx_1) \cong  L_0^{\oplus 2}$, $\mH^0(\Qx_2/\Qx_1) \cong L_0^{\oplus 3}$, $\mH^1(\Qx_2/\Qx_1) \cong L_\rho^{\oplus 2}$. Then the sequence \eqref{eqn-les-Omega-Tensor-T} becomes
\begin{equation}\label{eqn-les-Omega-Tensor-T-2}
0 \lra \mH^0(\Omega \otimes T) \lra L_0^{\oplus 3} \lra  L_0^{\oplus 2} \lra \mH^1(\Omega \otimes T) \lra  L_\rho^{\oplus 2} \lra 0 . 
\end{equation}

We also observe that $\Omega \otimes T \cong {\mathcal End}_{\Ox}(T)$ contains $\Ox$ as a direct summand. This corresponds to the fact that
$$\Omega \otimes T \cong {\mathcal End}_{\Ox}(T)\cong G\times_B(\n\otimes\lu),$$ 
and, as $B$-representations, $\n\otimes \lu \cong \mathrm{End}(\lu)$ contains $\C\cdot \mathrm{Id}_\lu$ as a direct summand. Now, Lemma~\ref{lemma-one-diml-section-End} applies and tells us that 
\begin{equation}\label{eqn-H-zero-Omega-Tensor-T}
\mH^0 (\Omega \otimes T) \cong L_0
\end{equation}

%

Then, plugging \eqref{eqn-H-zero-Omega-Tensor-T} into the long sequence \eqref{eqn-les-Omega-Tensor-T-2}, we have
\[  
 0 \lra L_0 \to L_0^{\oplus 3} \lra  L_0^{\oplus 2} \lra \mH^1(\Omega \otimes T) \lra  L_\rho^{\oplus 2} \lra 0 , 
\]
and therefore 
\begin{equation}\label{eqn-H-one-Omega-Tensor-T}
\mH^1(\Omega \otimes T) \cong  L_\rho^{\oplus 2}.
\end{equation}


\bigskip 
Now let us use Serre duality to compute the cohomology of  $\Omega^2 \otimes T$:
\begin{equation} \label{eqn-Serre-dual-Omega-Tensor-Omega}
\mH^i (\Omega^2 \otimes T) \cong \mH^{3-i}(\Omega \otimes \wedge^2 T \otimes \Omega^3)^* \cong \mH^{3-i}(\Omega \otimes \Omega)^* , 
\end{equation}
where we have used the isomorphism $\wedge^2 T \otimes \Omega^3 \cong \Omega$. We have $\Omega \otimes \Omega
\cong S^2(\Omega) \oplus  \Omega^2$, and the cohomology $\mH^i(\Omega^2)$ is known by Theorem~\ref{thm-flag-cohomology}. For $S^2 \Omega$, we will use Serre duality again:
\begin{equation} \label{eqn-Serre-dual-Sym2-Omega}
\mH^{3-i}(S^2\Omega)^* \cong \mH^i(S^2 T \otimes \Omega^3) .
\end{equation}
Taking the symmetric product of the filtration (\ref{T_filt}), and using the fact that the canonical bundle $\Omega^3\cong \Lx_{-2\alpha_1-2\alpha_2}$, we obtain the filtration of $S^2 T \otimes \Omega^3$ by vector subbundles:
\[
\mathcal{R}_1:=\Lx_{-2 \alpha_2} \oplus \Lx_{-2\alpha_1} \oplus \Lx_{-\alpha_1 -\alpha_2} \subset \mathcal{R}_2:=\mathcal{R}_1\oplus \Lx_{-\alpha_1} \oplus \Lx_{-\alpha_2} \subset \mathcal{R}_3:= \mathcal{R}_2\oplus \Lx_0 .
\]
Analyzing the weights by Theorem \ref{BWB}, we notice that the subbundle $\mathcal{R}_1$ contributes nothing to cohomology, and all cohomology of $S^2 T \otimes \Omega^3$ in degrees greater than 1 
vanishes.

By Theorem \ref{BWB}, $\mH^0(\mathcal{R}_2) =0$, $\mH^1(\mathcal{R}_2) \cong L_0^{\oplus 2}$, $\mH^0(\Lx_0) \cong L_0$, and $\mH^1(\Lx_0) =0$. Then we have
\begin{equation}  
0 \lra \mathcal{R}_2 \lra S^2 T \otimes \Omega^3  \lra \Lx_0 \lra 0 ,   \label{S^2TK} 
\end{equation}
which induces the sequence of cohomology groups
\[ 
0 \lra \mH^0(S^2 T \otimes \Omega^3) \lra \mH^0(\Lx_0) \lra \mH^1(\mathcal{R}_2) \lra \mH^1(S^2 T \otimes \Omega^3) \lra 0 .
\]
Equivalently, we have
\begin{equation}\label{eqn-les-S2-Omega-tensor-K} 
0 \lra \mH^0(S^2 T \otimes \Omega^3) \lra L_0 \lra  L_0^{\oplus 2} \lra \mH^1(S^2 T \otimes \Omega^3) \lra 0 .
\end{equation}

We claim that $\mH^0(S^2 T \otimes \Omega^3) \cong 0$. Otherwise, it would be isomorphic to $L_0$ by \eqref{eqn-les-S2-Omega-tensor-K}, which in turn means that the bundle has a $G$-equivariant global section splitting the projection map $S^2 T \otimes \Omega^3  \lra \Lx_0$. However, observe that the sequence (\ref{S^2TK}) is non-split. Indeed, the existence of a nonzero map $\Lx_0 \lra S^2T \otimes \Omega^3$ would imply the existence of a map on the level of the corresponding $B$-modules, but the filtration for $S^2T \otimes \Omega^3$
shows that it corresponds to a cyclic $B$-module generated by a single highest-weight-$0$ vector. Therefore, we conclude that $\mH^0(S^2 T \otimes \Omega^3) =0$, and then the cohomology sequence \eqref{eqn-les-S2-Omega-tensor-K} implies that $\mH^1(S^2 T \otimes \Omega^3) \cong L_0$.

It follows from this discussion that we have
\begin{align}
\mH^0(\Omega^2 \otimes T) & \cong \mH^3(\Omega \otimes \Omega)^* \cong \mH^3(S^2 \Omega)^* \oplus \mH^3(\Omega^2)^* \nonumber \\
& \cong
\mH^0(S^2 T \otimes \Omega^3) \oplus \mH^3(\Omega^2)^* \cong 0 \oplus 0 =0. \label{eqn-H0-Omega2-Tensor-T}
\end{align}
and
\begin{align} 
\mH^1(\Omega^2 \otimes T) & \cong \mH^2(\Omega \otimes \Omega)^* \cong \mH^2(S^2 \Omega)^* \oplus \mH^2(\Omega^2)^* \nonumber \\
& \cong
\mH^1(S^2 T \otimes \Omega^3) \oplus \mH^2(\Omega^2)^* \nonumber \\
& \cong L_0  \oplus  L_0^{\oplus 2} =   L_0^{\oplus 3} . \label{eqn-H1-Omega2-Tensor-T}
\end{align}

\paragraph{List of cohomology over $X$.}
For later use, we collect in a single table all results on cohomology groups of various sheaves over the three-dimensional flag variety $X=SL_3(\C)/B$.

\begin{equation}\label{table-1}
\begin{gathered}
\begin{array}{|c|c|c|c|c|}\hline
                              & \mH^0   & \mH^1    & \mH^2   & \mH^3    \\ \hline
 \Ox                          &  L_0    &  0       &   0     &  0       \\ \hline
 \Omega                       &  0      &  L_0 ^{\oplus 2}    &   0     &  0       \\ \hline
 \Omega^2                     &  0      &  0       &  L_0^{\oplus 2}   &  0       \\ \hline
 \Omega^3                     &  0      &  0       &   0     &  L_0     \\ \hline
 T                            &  L_\rho &  0       &   0     &  0       \\ \hline
 \Omega \otimes T             & L_0     &  L_\rho^{\oplus 2} &   0     &  0       \\ \hline
 \Omega^2 \otimes \wedge^2 T  & L_0     &  L_\rho^{\oplus 2} &   0     &  0       \\ \hline
 \Omega^2 \otimes T           &  0      &   L_0^{\oplus 3}   &   0     &  0       \\ \hline
 \end{array}
 \end{gathered} 
 \end{equation}

%
%

\subsection{The \texorpdfstring{$\mathfrak{sl}_3$}{sl(3)}-center computation} \label{sec-sl3}
In this subsection, we will use Theorem \ref{Hoh},
\[ 
\op{HH}^0 (\ul_0 ) \cong \bigoplus_{i+j+k=0} \mH^i(\Nt,  \wedge^j T\Nt)^k ,
\]
to describe the center of the principal block $\ul_0$ of the small quantum group for $\g= \s_3$.
Since $\Nt$ is six dimensional,  the index $j$ changes from $0$ to $6$.

Following the computation of the $\lb$-module structure corresponding to $\pr_*(\wedge^{\star}T\Nt)$ given in Section \ref{sec-Springer}, we could have used Theorem \ref{Bott_rel_lie} to compute their cohomology.  
In fact, in many cases we will manage to obtain the results by simpler explicit arguments presented below. However, the three crucial cases ($j=2,3,4$) benefit from the application of the relative Lie cohomology and the BGG complex.

We will use the short exact sequence of sheaves on $\Nt$, induced by the projection along the fiber $\op{pr}:\Nt\lra X:=G/B$:
\begin{equation}\label{vert_hor}
 0 \to T_{vert} \cong  \op{pr}^* \Omega_{X} \to T\Nt \to T_{hor} \cong \op{pr}^* T_{X} \to 0 , 
 \end{equation}
where $T_{vert}$ consists of vectors tangent to the (vertical) fiber direction and
$T_{hor}$ projects onto the tangent bundle of the base (c.~f.~the sequence (\ref{eqn-ses-two-on-N})). The $k$-degrees of the components are
$\op{deg}_k(\Omega_{X}) =-2$, $\op{deg}_k(T_{X}) = 0.$

\paragraph{Case $j=0$.} Then $i=k=0$ and we have
\[ 
\op{HH}^0(\ul_0)_{j=0} \cong \mH^0(\Nt, \Ox_\Nt)^0 \cong \mH^0(X, \Ox_X) \cong L_0. 
\]

\paragraph{Case $j=1$.}  We have to compute $\op{HH}^0(\ul_0)_{j=1} \cong \oplus_{i+k=-1} \mH^i(\Nt, T\Nt)^k$.
The short exact sequence \eqref{vert_hor} shows that the only possible value of $k$ in this case is $k=-2$
and it corresponds to the subsheaf $\op{pr}^* \Omega_{X} \subset T\Nt$. Then $i=1$ and we have
\[ 
\op{HH}^0(\ul_0)_{j=1} \cong \mH^1(\Nt, T\Nt)^{-2} \cong \mH^1(X, \Omega_X) \cong  L_0^{\oplus 2} .
\]

\paragraph{Case $j=2$.} We have to find $\op{HH}^0(\ul_0)_{j=2} \cong \oplus_{i+k=-2} \mH^i(\Nt,  \wedge^2 T\Nt)^k$.
The admissible values of $k$ are $k=-2$ and $k=-4$. If $k=-4$, then $i=2$ and the sheaf is the exterior
square of $T_{vert}$:
\[ 
\mH^2(\Nt, \wedge^2 T\Nt)^{-4} \cong \mH^2(X, \Omega_X^2) \cong  L_0^{\oplus 2} .
\]
When $k =-2$, $i=0$, in order to find 
$$\mH^0(\Nt,\wedge^2T\Nt)^{-2}\cong \mH^0(X, \mathrm{pr}_*(\wedge^2 T\Nt)^{-2} ),$$ 
we will compute the cohomology of $\mathcal{F}_1:=\mathrm{pr}_*(\wedge^2 T\Nt)^{-2} $ on $X$. In this case, two subquotients of $\mathcal{F}_1$ can contribute to the cohomology: $\Omega_{X}^2 \otimes T_X$ and
$\Omega_{X} \otimes T_{X}$. Indeed we have the following sequence of vector bundles over $X$:
\begin{equation}\label{eqn-F1-sheaf-sequence}
0 \lra  \Omega^2 \otimes T \lra \F_1 \lra \Omega \otimes T \lra 0 .
\end{equation} 
This induces a long exact sequence of cohomology groups over $X$:
\[  0 \to \mH^0(\Omega^2 \otimes T) \to \mH^0(\F_1) \to \mH^0(\Omega \otimes T) \to \mH^1(\Omega^2 \otimes T) \to
\mH^1(\F_1) \to \mH^1(\Omega \otimes T) \to  0 .\]
Plugging in the terms from Table \eqref{table-1}, we are left with the exact sequence
\begin{equation}\label{eqn-les-F1}
 0 \lra \mH^0(\F_1) \lra L_0 \lra  L_0^{\oplus 3} \lra \mH^1(\F_1) \lra  L_\rho^{\oplus 2} \lra 0 , 
\end{equation}
which tells us that $\mH^0(\F_1)$ is either zero or isomorphic to $L_0$ as a $G$-module. 

To determine which case it is, we need a more careful study of the equivariant sheaf $\F_1$. Recall from Corollary \ref{cor-exterior-product}, and in particular formula \eqref{eqn-V2-minus-2}, that 
\[
\F_1\cong G\times_B V_2^{-2}, 
\]
where $V_2^{-2}$ is the $B$-module (in fact a $B\times \C^*$-module, where the superscript $-2$ indicates the module has $\C^*$-weight or degree equal to $-2$)
\[
 V_2^{-2}=\dfrac{\g\otimes \n \oplus \lu \otimes (\n\wedge \n)}{\Delta_2(\lb\otimes \n)}. 
\]
Here the map $\Delta_2$ is given as the composition
\[ 
\Delta_2: \lb \otimes \n \xrightarrow{(\iota,\ad_\n)\otimes \mathrm{Id}_{\n}} \g\otimes \n \oplus \lu \otimes \n\otimes \n \stackrel{\beta}{\lra} \g\otimes \n \oplus \lu \otimes (\n\wedge \n)  .
\] 
The sequence \eqref{eqn-F1-sheaf-sequence} comes from the corresponding short exact sequence of $B$-modules
\begin{equation}\label{eqn-F1-B-module-sequence}
0 \longrightarrow \lu\otimes \n \wedge \n \lra V_2^{-2}\lra \dfrac{\g\otimes \n}{\lb\otimes \n}\cong \lu \otimes \n \lra 0.
\end{equation}
The cohomology $\mH^0(X,\Omega\otimes T)\cong \mH^0(X,G\times_B(\lu\otimes \n))\cong L_0$ comes from the splitting of the $B$-module
\[
\lu \otimes \n \cong \mathrm{End}(\n)\cong \mathbb{C} \mathrm{Id}_{\n} \oplus \mathfrak{sl}(\n), 
\]
where $\mathfrak{sl}(\n)$ stands for the space of traceless endomorphisms of $\n$. The subspace $\C\mathrm{Id}_\n$ spans a trivial $B$-submodule, and upon differentiation, a trivial $\lb$-submodule. To determine whether $\mH^0(\F_1)$ is nonvanishing or not, we need to find out whether the split inclusion 
$\C \mathrm{Id}_\n \subset \lu\otimes \n$
lifts to a trivial $B$-summand in $V_2^{-2}$:
\[
\begin{gathered}
\xymatrix{
 & \C\mathrm{Id}_\n\ar@{^{(}->}[d]^{\oplus} \ar@{-->}_?[dl]\\
V_2^{-2} \ar[r]& \lu \otimes \n 
}
\end{gathered} .
\]

The map $\Delta_2$ can now be easily computed on the tensor product basis of $\lb\otimes \n$ chosen as in equations \eqref{eqn-Chevalley-basis-ef} and \eqref{eqn-Chevalley-basis-h}. For instance
\begin{equation}\label{eqn-delta2-h1-f1}
\Delta_2(
h_1\otimes f_1 
)  =  h_1\otimes f_1 + \op{ad}_{\n}(h_1)\wedge f_1= h_1\otimes f_1-e_2\otimes f_1\wedge f_2+e_3\otimes f_1\wedge f_3  ,
\end{equation}
\begin{equation}\label{eqn-delta2-h2-f2}
\Delta_2(
h_2\otimes f_2 
)  =  h_2\otimes f_2 + \op{ad}_{\n}(h_2)\wedge f_2= h_2\otimes f_2+e_1\otimes f_1\wedge f_2-e_3\otimes f_3\wedge f_2 .
\end{equation}

\begin{lemma}\label{lemma-splitting-V2-minus-2}The surjective composition map
\[
V_2^{-2}=\dfrac{\g \otimes \n\oplus \lu\otimes \n\wedge \n}{\Delta_2(\lb\otimes \n)}\lra (\g/\lb)\otimes \n \cong \mathrm{End}(\n)\lra \C\mathrm{Id}_\n
\]
splits as a map of $\lb$-modules, and the splitting summand is spanned by the element 
\[
z := e_1\otimes f_1 + e_2\otimes f_2 + e_3\otimes f_3 -e_3\otimes f_1\wedge f_2 
\]
modulo $\Delta_2(\lb\otimes \n)$.
\end{lemma}
\begin{proof}
It is clear that $z$ maps to the element $\mathrm{Id}_\n$ under the composition map, since by our normalization, $\langle e_i, f_j \rangle = \delta_{i,j}$ ($i,j =1,2,3$). It suffices to show that $z$ is annihilated by all elements of $\lb$. 

To prove this, notice that the weight of $z$ is zero. Thus $\h$ acts trivially on $z$. We are reduced to showing that $f_1$ and $f_2$ both kill $z$, and the result will follow since $f_3=[f_1,f_2]$.

We compute
\begin{eqnarray*}
f_1\cdot z &\!\! =\!\! & [f_1,e_1]\otimes f_1+e_1\otimes [f_1,f_1]+ [f_1,e_2]\otimes f_2 + e_2\otimes [f_1,f_2] +[f_1,e_3]\otimes f_3 +e_3\otimes [f_1,f_3] \\
& & - \left([f_1,e_3]\otimes f_1\wedge f_2 + e_3\otimes [f_1,f_1]\wedge f_2 + e_3\otimes f_1\wedge [f_1,f_2]\right) \\
& \!\! =\!\! & -h_1\otimes f_1+e_2\otimes f_1\wedge f_2 - e_3\otimes f_1\wedge f_3.
\end{eqnarray*}
It follows from equation \eqref{eqn-delta2-h1-f1} that
\[
f_1\cdot z = -\Delta_2 (h_1\otimes f_1),
\]
so that $f_1\cdot z \equiv 0~(\mathrm{mod}~\Delta_2(\lb\otimes \n))$. Likewise, we have
\begin{eqnarray*}
f_2\cdot z
& \!\! =\!\! & -h_2\otimes f_2-e_1\otimes f_1\wedge f_2 + e_3\otimes f_3\wedge f_2,
\end{eqnarray*}
which equals $-\Delta_2 (h_2\otimes f_2)$ by equation \eqref{eqn-delta2-h2-f2} and thus becomes zero in $V_2^{-2}$. The result now follows.
\end{proof}

Lemma \ref{lemma-splitting-V2-minus-2} tells us that the bundle $ \F_1\cong G\times_B V_2^{-2}$ does contain a $G$-equivariant trivial summand $\mathcal{O}_X$, and should thus have its space of global sections at least one dimensional. Combined with the sequence \eqref{eqn-les-F1}, we have the equality
\[ 
 \mH^0(\Nt, \wedge^2T\Nt)^{-2} \cong \mH^0(\F_1)   \cong L_0 . 
\]   

\begin{remark} The computation of $ \mH^0(\mathcal{F}_1)$ given above shows the explicit geometric and algebraic structures that contribute to the nontrivial cohomology. A more direct geometric meaning of the spanning global section in the general case will be given in Lemma \ref{lemma-Poisson}. Alternatively, we can compute this cohomology using the BGG resolution approach as described in  
Theorem \ref{Bott_rel_lie} and Proposition \ref{rel_lie_BGG}.  Since we need to compute only the zeroth cohomology, 
the first step is to find which dominant weights appear in $V_2^{-2}$.  We immediately see that the only dominant weight occurring in $V_2^{-2}$ is zero, 
and  that the subspace $V_2^{-2}[0]$ is spanned by the elements $\{e_1 \otimes f_1, e_2 \otimes f_2, e_3 \otimes f_3, e_3 \otimes f_1 \wedge f_2 \}$. 
Then we need to compute the zeroth cohomology of the complex  (\ref{F_BGG})  for $E = V_2^{-2}$ and $\lambda =0$: 
\[ 
0 \lra V_2^{-2}[0] \stackrel{(f_1, f_2)}{\longrightarrow}  V_2^{-2}[s_1 \cdot 0] \oplus V_2^{-2}[s_2 \cdot 0] \lra \cdots 
\] 
The dimension of the zeroth cohomology group of this complex is equal to the dimension of the intersection of the kernels of $f_1$ and $f_2$ acting on $V_2^{-2}[0]$. This is the  dimension of the subspace 
$S \subset (\lu \otimes \n)[0] \oplus(\lu \otimes \wedge^2 \n)[0]$ such that $f_1(S) \subset \Delta_2(\lb \otimes \n)$ 
and $f_2(S) \subset \Delta_2(\lb \otimes \n)$. Using the structure of the submodule $\Delta_2(\lb \otimes \n)$ 
(equations (\ref{eqn-delta2-h1-f1}) and (\ref{eqn-delta2-h2-f2}) suffice for our purposes), we compute that $\op{dim}(S)=1$, and $S$ is spanned by the cocycle $z$ in Lemma \ref{lemma-splitting-V2-minus-2}. 
\end{remark}

\paragraph{Case $j=3$.} We have to find $\op{HH}^0(\ul_0)_{j=3} \cong \oplus_{i+k=-3} \mH^i(\Nt,  \wedge^3 T\Nt)^k$.
The admissible values of $k$ are $k=-4$ and $k=-6$. 

If $k=-6$, then $i=3$. We have $\mathrm{pr}_*(\wedge^3 T\Nt)^{-6}\cong \Omega_X^3$, so that
\[
\mH^3(\Nt, \wedge^3 T\Nt)^{-6} \cong \mH^3(X, \Omega_X^3) \cong L_0 .
\]

When $k =-4$ and $i=1$,  let us consider the bundle $\mathcal{F}_2:=\mathrm{pr}_*(\wedge^3T\Nt)^{-4}$. Two subquotient bundles can contribute to the cohomology, namely $\Omega_{X}^3 \otimes T_X$ and
$\Omega^2_{X} \otimes T_{X}$, and they fit into the sequence
\[  
0 \lra \Omega^3 \otimes T \lra \F_2 \lra \Omega^2 \otimes T \lra 0 .
\]
Since $\Omega^3$ is the canonical bundle of $X$, using the isomorphism $\Omega^3 \otimes T \cong \Omega^2$,  we obtain
\[ 
0 \lra \Omega^2 \lra \F_2 \lra \Omega^2 \otimes T \lra  0 . 
\]
This induces the following long exact sequence in cohomology, part of which looks like
\begin{equation}\label{eqn-F2-les}
0 \lra \mH^1(\F_2) \lra \mH^1(\Omega^2 \otimes T) \lra \mH^2 (\Omega^2) \lra  \mH^2(\F_2) \lra \mH^2(\Omega^2 \otimes T) \lra \cdots .
\end{equation}
We are interested only in the term $\mH^1(\F_2)$.  Using the known cohomology groups of $\Omega^2$ and
$\Omega^2 \otimes T$ from Table \ref{table-1}, we have
\begin{equation} \label{coh_F2} 
0 \lra \mH^1(\F_2) \lra L_0^{\oplus 3} \lra  L_0^{\oplus 2} \lra \mH^2(\F_2) \lra 0 .
\end{equation} 
This shows that $\mH^1(\F_2)$ is  at least one dimensional and may be isomorphic to $m$ copies of the trivial $G$-module, with $1 \leq m \leq 3$. To compute the multiplicity $m$, we will use Theorem \ref{Bott_rel_lie} and Proposition \ref{rel_lie_BGG}.

By Corollary \ref{cor-exterior-product} and formula \eqref{eqn-V3-minus-4}, we know that the $\lb$-module structure corresponding to the sheaf $\op{pr}_*(\wedge^3 T(\Nt))^{-4}$ is given by 
\[ 
V_3^{-4}=\dfrac{\g\otimes (\n \wedge \n)\oplus \lu \otimes (\n\wedge \n\wedge \n)}{\Delta_3(\lb\otimes (\n\wedge \n))},
\] 
where 
\[ \Delta_3 : \lb \otimes (\n\wedge \n) \xrightarrow{(\iota,\ad_\n)\otimes \mathrm{Id}_{\n\wedge\n}} \g\otimes (\n\wedge \n) \oplus \lu \otimes \n\otimes (\n\wedge\n) \stackrel{\beta}{\lra} \g\otimes (\n\wedge \n) \oplus \lu \otimes (\n\wedge \n \wedge \n). 
   \]  
Then Theorem \ref{Bott_rel_lie} states that  
\[ \op{Hom}_G(L_0, \mH^1(X, \op{pr}_*(\wedge^3T\Nt)^{-4}) \cong \op{H}^1(\lb, \h, V_3^{-4}) . \]

To compute this relative Lie algebra cohomology, we need to find the dominant  weights $\lambda$ such that  $s_1 \cdot \lambda$ or $s_2 \cdot \lambda$ is a weight  in $V_3^{-4}$.   
The only weight with this property is $\lambda =0$. The weight subspaces of weights $w \cdot 0$ for $w = \{ 1, s_1, s_2 \}$  
are spanned by the following vectors (modulo  $\Delta_3(\lb \otimes (\n \wedge \n))$: 
\begin{itemize}
\item for $w=1$, $ \{ e_{3} \otimes f_1 \wedge f_2 \}$, 
\item for $w=s_1$, $\{ e_2 \otimes f_1 \wedge f_2, e_{3}  \otimes f_1 \wedge f_3 \}$,
\item for $w=s_2$, $\{  e_1 \otimes f_1 \wedge f_2,  e_{3} \otimes f_2 \wedge f_3 \}$. 
\end{itemize}  

 Then the complex (\ref{F_BGG}) 
 for $E = V_3^{-4}$ and $\lambda =0$ becomes 
 \begin{equation} \label{BGG-V_34}
 0 \lra V_3^{-4}[0] \stackrel{d_1^*}{\lra}V_3^{-4}[s_1 \cdot 0] \oplus V_3^{-4} [s_2 \cdot 0]   
 \stackrel{d_2^*}{\lra} V_3^{-4}[s_2 s_1 \cdot 0] \oplus V_3^{-4} [s_1 s_2 \cdot 0]  \lra \cdots  \ .
 \end{equation} 
 
 \begin{lemma} The first cohomology group of the complex (\ref{BGG-V_34}) is three dimensional. 
 \end{lemma} 
 \begin{proof} 
The maps $d_1^*$ and $d_2^*$  in the complex (\ref{BGG-V_34}) are given by the diagram \eqref{eqn-BGG-sl3}. In particular, $d_1^* = f_1 \oplus f_2$, and 
\[ 
\left. d_2^*\right|_{V_3^{-4}[s_1 \cdot 0]}  = f_2^2 \oplus (- 2 f_1 f_2 + f_2 f_1) , \quad  \quad \quad 
\left. d_2^*\right|_{V_3^{-4} [s_2 \cdot 0]} = f_1^2 \oplus (-2 f_2f_1 + f_1 f_2 ).
\]  

It is easy to see that the image of the map $d_1^*$ acting on $V_3^{-4}[0]$ is one dimensional. Indeed, we have  
\[ f_1(e_3 \otimes f_1 \wedge f_2) = -e_2 \otimes f_1 \wedge f_2 + e_3 \otimes f_1 \wedge f_3 , \] 
\[ f_2( e_3 \otimes f_1 \wedge f_2) = e_1 \otimes f_1 \wedge f_2 - e_3 \otimes f_3 \wedge f_2  . \] 
These elements are not in the submodule $\Delta_3(\lb \otimes (\n \wedge \n))$ because by construction 
this submodule   does not intersect the subspace 
$\lu \otimes  \n \wedge \n$. 

To find the kernel of $d_2^*$, we compute  
\[ \begin{array}{l}  
f_1^2 (e_1 \otimes f_1 \wedge f_2) = -2f_1 \otimes f_1 \wedge f_2 - 2 h_1 \otimes f_1 \wedge f_3, \\
f_1^2 (e_3  \otimes f_2 \wedge f_3) =0,  \\
f_2^2 (e_2 \otimes f_1 \wedge f_2) = -2 f_2 \otimes f_1 \wedge f_2 + 2 h_2 \otimes f_3 \wedge f_2, \\
 f_2^2 (e_3 \otimes f_1 \wedge f_3) =0,  \\
(-2f_2 f_1 +f_1 f_2)(e_1 \otimes f_1 \wedge f_2) = - 2 f_2 \otimes f_1 \wedge f_2 - h_1 \otimes f_3 \wedge f_2,  \\   
(-2f_2 f_1 +f_1 f_2)(e_3 \otimes f_2 \wedge f_3) = -h_1 \otimes f_2 \wedge f_3 - 2h_2 \otimes f_2 \wedge f_3,   \\ 
(- 2 f_1 f_2 + f_2 f_1) (e_2 \otimes f_1 \wedge f_2) = -2 f_1 \otimes f_1 \wedge f_2 + h_2 \otimes f_1 \wedge f_3, \\ 
(- 2 f_1 f_2 + f_2 f_1)(e_3 \otimes f_1 \wedge f_3)  = h_2 \otimes f_1 \wedge f_3 + 2 h_1 \otimes f_1 \wedge f_3.  
 \end{array}  \]
The right-hand side terms of all these equations lie in the submodule $\Delta_3(\lb \otimes (\n \wedge \n))$. 
Indeed, we have  
\[ 
\begin{array}{rcl}  
\Delta_3(h_1 \otimes f_3 \wedge f_1) & = & h_1 \otimes f_3  \wedge f_1   + e_2  \otimes f_1 \wedge f_2 \wedge f_3, \\
\Delta_3(h_2 \otimes f_1 \wedge f_2) & = &  h_2 \otimes f_1  \wedge f_2  -e_3  \otimes f_1 \wedge f_2 \wedge f_3,  \\ 
\Delta_3(h_2 \otimes f_2  \wedge f_3) & =  & h_2 \otimes f_2  \wedge f_3  + e_1  \otimes f_1 \wedge f_2 \wedge f_3, \\ 
\Delta_3(h_2 \otimes f_3 \wedge f_1) & = & h_2 \otimes f_3  \wedge f_1   - 2e_2  \otimes f_1 \wedge f_2 \wedge f_3, \\
\Delta_3(f_1 \otimes f_1 \wedge f_2) & = & f_1 \otimes f_1  \wedge f_2  +e_2  \otimes f_1 \wedge f_2 \wedge f_3,  \\ 
\Delta_3(f_2 \otimes f_1 \wedge f_2) & = & f_2 \otimes f_1  \wedge f_2  -e_1 \otimes f_1 \wedge f_2 \wedge f_3 .  \\ 
\end{array} 
\]   
This implies the equalities 
\[ 
f_1 \otimes f_1 \wedge f_2 \equiv h_1 \otimes f_3\wedge f_1 \equiv -\frac{1}{2} h_2 \otimes f_3 \wedge f_1 \quad (\op{mod}~   
\Delta_3(\lb \otimes (\n \wedge \n)), 
\] 
\[ 
f_2 \otimes f_1 \wedge f_2 \equiv  - h_2 \otimes f_2 \wedge f_3 \equiv  \frac{1}{2} h_1 \otimes f_2 \wedge f_3 \quad (\op{mod}~   
\Delta_3(\lb \otimes (\n \wedge \n)). 
\] 
Thus we obtain that the kernel of $d_2^*$ in the complex (\ref{BGG-V_34}) is four dimensional, and its first cohomology group is three dimensional.  
\end{proof}  

Therefore the cohomology $\op{H}^1(\lb, \h, V_3^{-4})$ is three dimensional and we have  $\mH^1(\F_2) \cong L_0^{\oplus 3}$. 
 According to the sequence \eqref{coh_F2}, this implies that $\mH^2(\F_2) \cong L_0^{\oplus2}$, a result that can be confirmed explicitly by computing the cohomology of the next term of the complex (\ref{BGG-V_34}). 
 
 Finally, we have obtained the component of the center 
\[ 
\mH^1(\Nt, \wedge^3T\Nt)^{-4} \cong L_0^{\oplus 3} .
\]

\paragraph{Case $j=4$.} The computations in this case can be reduced to the $j=2$ case via Corollary \ref{cor-half-reduction}. However, for the sake of completeness, we also record a direct computation here.

We have to find $\op{HH}^0(\ul_0)_{j=4} \cong \oplus_{i+k=-4} \mH^i(\Nt,  \wedge^4 T\Nt)^k$.
The admissible values of $k$ are $k=-4$ and $k=-6$. 

If $k=-6$, then $i=2$ and the degree $-6$ part of $\op{pr}_*(\wedge^4T\Nt)$ is equal to
$\Omega_X^3\otimes T_X$:
\[ 
\mH^2(\Nt, \wedge^4 T\Nt)^{-6} \cong \mH^2(X, \Omega_X^3 \otimes T_X) \cong \mH^2(X, \Omega_X^2) \cong L_0^{\oplus 2} .
\]
In the second isomorphism we have used the identification of sheaves over the three-dimensional variety $X$: 
$$\Omega_X^3 \otimes T_X \cong \Omega_X^2.$$

When $k =-4$, $i=0$, let $\mathcal{F}_3:=\mathrm{pr}_*(\wedge^4T\Nt)^{-4}$, which fits into the following short exact sequence of vector bundles over $X$:
\begin{equation}\label{eqn-ses-F3}
0 \lra \Omega^3 \otimes T \otimes T \lra \F_3 \lra \Omega^2 \otimes \wedge^2 T \lra 0 .
\end{equation}
Using isomorphism of vector bundles $\Omega^3 \otimes T \cong \Omega^2$, we get
\[ 
0 \lra \Omega^2 \otimes T \lra \F_3 \lra \Omega^2 \otimes \wedge^2 T \lra 0 .
\]
We then obtain the long exact sequence of cohomology groups
\[ 
0 \lra \mH^0(\Omega^2\otimes T)\lra \mH^0(\F_3) \lra \mH^0(\Omega^2\otimes \wedge^2 T)\lra \mH^1(\Omega^2\otimes T)\lra \mH^1(\F_3)\lra \cdots,
\]
which further reduces to (Table \eqref{table-1})
\[
0 \lra \mH^0(\F_3)\lra L_0\lra L_0^{\oplus 3}\lra  \mH^1(\F_3) \lra \cdots.
\]
We are again left to find out whether the sequence \eqref{eqn-ses-F3} splits equivariantly. 

To determine whether the splitting happens, we need a more detailed understanding of the bundle 
$\F_3\cong G\times_B V_4^{-4}$, where 
\begin{equation}\label{eqn-V4-minus-4}
V_4^{-4}=(\wedge^4V_1)^{-4}\cong
\dfrac{\wedge^2 \g \otimes \wedge^2\n \oplus \lu \otimes \g\otimes \wedge^3 \n}{\Delta(\lb)\wedge(  \g \otimes \wedge^2 \n\oplus \lu\otimes \wedge^3\n)}.
\end{equation}

\begin{lemma} In terms of the Chevalley basis for $\mathfrak{sl}_3$ chosen in \eqref{eqn-Chevalley-basis-ef} and \eqref{eqn-Chevalley-basis-h}, the element
$$w:=e_1\wedge e_2 \otimes f_1\wedge f_2+e_1\wedge e_3 \otimes f_1\wedge f_3+e_2\wedge e_3 \otimes f_2\wedge f_3-e_3\otimes e_3\otimes f_1\wedge f_2 \wedge f_3$$
spans a trivial summand in the $\lb$-module $V_4^{-4}$. 
\end{lemma}
\begin{proof}The proof is similar to that of Lemma \ref{lemma-splitting-V2-minus-2}. One is reduced to checking the following identities
\[
f_1\cdot w= -\Delta(h_1)\wedge (e_2\otimes f_1\wedge f_2+e_3\otimes f_1\wedge f_3)),
\]
\[
f_2\cdot w= -\Delta(h_2)\wedge (e_1\otimes f_1\wedge f_2+e_3\otimes f_2\wedge f_3)).
\]
\end{proof}

\begin{remark} The element $w$ is proportional to the second wedge power of the element $z$ found in Lemma 
\ref{lemma-splitting-V2-minus-2}. A more general description of $z$ and its powers will be given 
in Lemma \ref{lemma-Poisson} of Section \ref{sec-center-symmetry}. 
\end{remark}

\paragraph{Case $j=5$.} In this case, $\op{HH}^0(\ul_0)_{j=5} \cong \oplus_{i+j =-5} \mH^i(\Nt, \wedge^5 T\Nt)^k$
and the only admissible value of $k$ is $k=-6$, corresponding to the sheaf $\wedge^3 T_{vert} \otimes \wedge^2 T_{hor}$.
We have
\[ 
\op{HH}^0(\ul_0)_{j=5} \cong \mH^1(\Nt,  \op{pr}^*(\Omega^3_{X} \otimes \wedge^2 T_{X} ) )
\cong \mH^1(X, \Omega_X^3 \otimes \wedge^2 T_X) \cong \mH^1(X, \Omega_X) \cong  L_0^{\oplus 2} .
\]

\paragraph{Case $j=6$.} Then $k=-6$, $i=0$ and for rank reasons the sheaf $\wedge^6 T\Nt\cong \wedge^3 T_{vert} \otimes \wedge^3 T_{hor}$. We have
\[ 
\op{HH}^0(\ul_0)_{j=6} \cong \mH^0(\Nt, \op{pr}^*(\Omega^3_{X} \otimes \wedge^3 T_{X})) \cong
\mH^0(X, \Ox_X) \cong L_0. 
\]

\paragraph{Summary.} 
Let us introduce the following notion for any semisimple Lie group $G$ over a characteristic-zero algebraically closed field.

As in the $\mathfrak{sl}_3$ case, let $X=G/B$ be the flag variety whose complex dimension equals $n$, and $\Nt:=T^*X$ be the Springer variety.   
\begin{definition}\label{def-formal-Hodge}
The \emph{formal Hodge diamond}\footnote{Here we keep in mind that $\Omega_{\Nt}\cong T_{\Nt}$.} for the variety $\Nt$ is the bigraded zeroth Hochschild cohomology table
\begin{equation}\label{table-Hodge}
\begin{gathered}
\begin{array}{|c|c|c|c|} 
\hline
 \mH^0(\wedge^0T\Nt)^0       &                             &                         &                            \\  \hline 
 \mH^1(\wedge^1T\Nt)^{-2}    &  \mH^0(\wedge^2T\Nt)^{-2}   &                         &                            \\  \hline 
         \vdots              &        \vdots               &  \hspace{0.3in}\ddots \hspace{0.3in}    &                   \\  \hline
 \mH^n(\wedge^nT\Nt)^{-2n}   &  \mH^{n-1}(\wedge^{n+1}T\Nt)^{-2n}   &  \cdots                 &  \mH^0(\wedge^{2n}T\Nt)^{-2n}    \\  \hline 
\end{array}  
\end{gathered} \ .
\end{equation}
The empty boxes indicate that the corresponding spaces vanish due to degree reasons. The dimension of each entry will be denoted by
\[
h^{i,j}:= \mathrm{dim}_\C(\mH^{i}(\wedge^jT\Nt)^{-i-j}).
\]
\end{definition}

\begin{theorem}\label{thm-sl3}
The dimension table for the formal Hodge diamond for small quantum $\mathfrak{sl}_3$ is given by 
\[
\begin{gathered}
\begin{array}{|c||c|c|c|c|} \hline 
         \scriptstyle{ j+i=0 } & 1   &                         &                         &                          \\  \hline 
              \scriptstyle{ j+i=2 } & 2   &   1                 &                         &                          \\  \hline 
               \scriptstyle{ j+i=4 }   & 2   &  3  &    1  &                          \\  \hline
             \scriptstyle{ j+i=6 }    &  1   &  2  &  2  &   1   \\  \hline \hline 
 \scriptstyle{h^{i,j}} &  \scriptstyle{ j-i=0 }       &     \scriptstyle{ j-i=2 }       &    \scriptstyle{ j-i=4 }    &  \scriptstyle{ j-i=6 }  
 \\ \hline 
\end{array}  
\end{gathered} \ .
\]
In particular, the center of the principal block for $u_q(\mathfrak{sl}_3)$ is $16$ dimensional. Furthermore, each entry $\mH^i(\wedge^j T\Nt)^{-i-j}$ in the formal Hodge diamond is a direct sum of trivial $\mathfrak{sl}_3$-representations.

\end{theorem}

The table suggests that there is a bigraded isomorphism of vector spaces between the formal Hodge diamond for $\g = \mathfrak{sl}_3$ and the diagonal coinvariant algebra for $S_3$ \cite{Hai} (the general definition will be recalled in Section \ref{sec-conj}): 
\[
\mathrm{DC}_3:=\dfrac{\C[x_1,x_2,x_3, y_1,y_2, y_3]}{\C[x_1,x_2,x_3, y_1,y_2, y_3]^{S_3}_+}.
\]
When equipped with the bigrading $\mathrm{deg}(x_i)=(1,0)$ and $\mathrm{deg}(y_i)=(0,1)$ for all $i=1,2,3$, the bigraded dimension table for $\mathrm{DC}_3$ ($d^{i,j}:=\mathrm{dim}(\mathrm{DC}_3^{i,j})$) is 
\[
\begin{gathered}  
\begin{array}{|c||c|c|c|c|}   \hline    
     \; \;     \scriptstyle{ i=3 }  \; \;  & 1   &                         &                         &                          \\  \hline 
       \; \;   \scriptstyle{ i=2 }   \; \;  & 2   &   1                 &                         &                          \\  \hline 
        \; \;  \scriptstyle{ i=1 }  \; \;   & 2   &  3  &    1  &                          \\  \hline
   \; \;  \scriptstyle{ i=0 }   \; \;     &  1   &  2  &  2  &   1   \\  \hline \hline 
 \scriptstyle{d^{i,j}} & \;\;  \scriptstyle{j= 0 }   \;\;    &    \;\;  \scriptstyle{ j=1 }   \;\;    &  \;\; \scriptstyle{ j=2 }   \;\;  &  \;\;  \scriptstyle{ j=3 }   \;\; 
   \\ \hline 
\end{array}  
\end{gathered} \ .
\]

In particular, we have the equality 
\begin{equation}
d^{i,j}=h^{3-i-j,3+j-i},
\end{equation}
or equivalently
\begin{equation}
h^{i,j}=d^{3-\frac{i+j}{2}, \frac{j-i}{2}}.
\end{equation}

\begin{remark} 
The striking similarity between the formal Hodge diamond and the diagonal coinvariant algebra, in hindsight, is already evident 
(although not easily guessed!) for the principal block center $\zl_0(\mathfrak{sl}_2)$ from the work of Kerler \cite{Ker}, who computed it to be three dimensional. Using the method we have developed in this paper, it is easy to see that 
\begin{equation}\label{table-Hodge-sl2} 
\begin{gathered} 
\begin{array}{|c||c|c|} 
\hline    
         \scriptstyle{ j+i=0 }  & 1   &                                               \\  \hline 
        \scriptstyle{ j+i=2 }    & 1   &   1                                   \\  \hline \hline 
 \scriptstyle{h^{i,j}}  &   \scriptstyle{ j-i=0 }      &     \scriptstyle{ j-i= 2 }       \\ \hline  
   \end{array} 
\end{gathered} \ , 
\end{equation}
while the diagonal coinvariant algebra 
$$\mathrm{DC}_2:=\frac{\C[x_1,x_2,y_1,y_2]}{\C[x_1,x_2,y_1,y_2]^{S_2}_+}$$ 
has its bigraded dimension table equal to 
\begin{equation}
\begin{gathered}
\begin{array}{|c||c|c|} 
\hline 
        \scriptstyle{ i=1 }   & 1   &                                               \\  \hline 
   \;\;\; \scriptstyle{ i=0 }   \;\;\;   & 1   &   1                                   \\  \hline \hline 
    \scriptstyle{d^{i,j}}  & \;\;  \scriptstyle{ j=0 }  \;\;    &    \;\;  \scriptstyle{ j=1 }  \;\;     \\ \hline  
\end{array} 
\end{gathered} \ .
\end{equation}
A conjecture generalizing these similarities between the principal block of the center and the diagonal coinvariant algebra will be formulated in Section \ref{sec-conj}.
\end{remark}

\section{Symmetries of the center}\label{sec-center-symmetry}
\subsection{An \texorpdfstring{$\mathfrak{sl}_2$}{sl(2)}-action on the center}
Corollary \ref{cor-half-reduction} shows that the formal Hodge diamond \eqref{table-Hodge} has a $\Z/(2)$-symmetry of reflecting the entire table about the anti-diagonal. In this section we will obtain additional symmetry results for the principal block of the center by constructing an $\mathfrak{sl}_2$ action on the formal Hodge diamond. This action is reminiscent of the usual $\mathfrak{sl}_2$-action on the Dolbeault cohomology of a compact K\"{a}hler manifold.

\begin{lemma}\label{lemma-Poisson}
Let $\g$ be a simple Lie algebra and $\Nt$ its associated Springer resolution, then 
$$\mH^0(\Nt, \wedge^2T\Nt)^{-2}\cong L_0,$$ 
and it is spanned by the holomorphic Poisson bivector field $\tau$ which is dual to the canonical symplectic form $\omega$ on $T\Nt$.
\end{lemma}
\begin{proof}
It is clear that the Poisson bivector field is a global section of $\wedge^2T\Nt$ that has degree $-2$. It is $G$-equivariant because its dual, the degree-2 symplecitic form $\omega$, is preserved under the infinitesimal $G$-action: for any $x\in \g$, the induced vector field $X\in \Gamma(\Nt, T\Nt)$ is Hamiltonian:
$
L_X(\omega) = 0.
$
It remains to give an upper bound for the dimension of $\mH^0(\Nt,\wedge^2\Nt)^{-2}$.

The bundle $\mathrm{pr}_*(\wedge^2 T\Nt)^{-2}$ fits into a short exact sequence on $X$
\[
0 \lra T\otimes \Omega^2 \lra \mathrm{pr}_*(\wedge^2T\Nt)^{-2} \lra T\otimes \Omega \lra 0.
\]
Taking sections, we get
\[
0\lra \mH^0(X,T\otimes \Omega^2 )\lra \mH^0(\Nt, \wedge^2\Nt)^{-2}\lra \mH^0(X, T\otimes \Omega)\lra \cdots.
\]
Using Lemma \ref{lemma-one-diml-section-End} that $\mH^0(X, T\otimes \Omega)\cong \C$, the claim will follow from the next vanishing result.
\end{proof}

\begin{lemma}
Let $\g$ be a simple Lie algebra and $X$ be its associated flag variety. Then
\[
\mH^0(X, T_X\otimes \Omega^2_X)=0.
\]
\end{lemma}
\begin{proof}
In the proof, we identify $T_X\cong G\times_B\lu$ and $\Omega_X\cong G\times_B \n$, and abbreviate the corresponding cohomology as $\mH^\bullet(\lu):=\mH^\bullet(X,T_X)$ etc. Consider the short exact sequence of vector bundles
\[
G\times_B\left( 0\lra \wedge^2 \n\otimes \lb \lra \wedge^2\n \otimes \g \lra \wedge^2\n\otimes \lu \lra 0 \right).
\]
Taking cohomology, we get that
\[
0\cong \mH^0(\wedge^2\n\otimes\g)\lra \mH^0(\wedge^2\n\otimes\lu)\lra \mH^1(\wedge^2\n\otimes \lb),
\]
and it suffices to show that $\mH^1(\wedge^2\n\otimes \lb)=0$. 

To do this, we use a second short exact sequence
\[
G\times_B\left( 0\lra \wedge^2 \n\otimes \n \lra \wedge^2\n \otimes \lb \lra \wedge^2\n\otimes \h \lra 0 \right)
\]
to bound the interested-in $\mH^1(\wedge^2\n\otimes \n)$:
\[
0\cong \mH^0(\wedge^2\n \otimes \h)\lra \mH^1(\wedge^2\n \otimes \n)\lra \mH^1(\wedge^2\n\otimes \lb)\lra \mH^1(\wedge^2\n \otimes \h)\cong 0.
\]
Therefore, the desired vanishing result will follow if we show $\mH^1(\wedge^2\n\otimes \n)=0$. This can be done by analyzing the weights in the module $\wedge^2\n\otimes\n$, which are of the form $\lambda=-\alpha-\beta-\gamma$, where $\alpha,\beta, \gamma$ are positive roots such that $\alpha$ and $\beta$ are distinct. For such a weight to contribute to the first cohomology, we need the shifted simple reflection action on $\lambda$
\[
s_i\cdot \lambda = -(\alpha+\beta+\gamma)+(\langle \alpha+\beta+\gamma,\alpha_i^\vee\rangle -1)\alpha_i
\]
to be a dominant weight, where $\alpha_i$ is some simple root, and $s_i$ is the corresponding simple refection. Since the dominant weight chamber is contained in the positive root cone, this can never happen as $\alpha$, $\beta$ are distinct.
\end{proof}

Consider the following $\mathcal{O}_\Nt$-linear bundle operations on the total exterior product bundle $\wedge^\star T\Nt$: given a local section of $\wedge^\star T\Nt$, define
\begin{equation}\label{eqn-wedge-tau}
\tau\wedge(-):\wedge^\star T\Nt {\lra} \wedge^{\star+2} T\Nt,
\quad
\eta \mapsto \tau \wedge \eta,
\end{equation}\label{eqn-contract-omega}
to be the wedge product with the Poisson bivector field, and
\begin{equation}
\iota_\omega(-):\wedge^\star T\Nt {\lra} \wedge^{\star-2} T\Nt,
\quad
\eta \mapsto \iota_\omega  (\eta),
\end{equation}
to be the contraction map with the symplectic form $\omega$. 

\begin{theorem}\label{thm-sl2-action}The two maps $\tau\wedge$ and $\iota_\omega$ satisfy the relation
\[
[\iota_\omega, \tau\wedge]=n-\mathrm{deg}: \wedge^\star T\Nt \lra \wedge^\star T\Nt.
\]
Consequently, they induce an $\mathfrak{sl}_2$-action on the Hochschild cohomology groups
\[
\tau\wedge: \mH^i(\Nt, \wedge^jT\Nt)^k\lra \mH^i(\Nt, \wedge^{j+2}T\Nt)^{k-2}, 
\]
\[ 
\iota_\omega: \mH^i(\Nt, \wedge^jT\Nt)^k\lra \mH^i(\Nt, \wedge^{j-2}T\Nt)^{k+2}.
\]
\end{theorem}
\begin{proof}Since multiplication by $\tau$ and contraction with $\omega$ are $\mathcal{O}_{\Nt}$-linear, it suffices to check the commutator relation on any fiber of the bundle $\wedge^\star T\Nt$ at a point $p\in \Nt$.

Locally, if we choose a symplectic coordinates system $\{ x_i, y_i|i=1,\dots, n\}$ near a point $p\in \Nt$ such that
\[
\omega( \frac{\partial }{\partial x_i}, \frac{\partial }{\partial y_j})=-\omega( \frac{\partial }{\partial y_j}, \frac{\partial }{\partial x_i})=\delta_{i,j},
\]
then the Poisson bivector field can be written locally as
\[
\tau_p=\sum_{i=1}^n  \frac{\partial }{\partial x_i}\wedge  \frac{\partial }{\partial y_i},
\]
and the symplectic form $\omega$ equals
\[
\omega_p=\sum_{i=1}^n dx_i\wedge dy_i.
\]
Since distinct pairs of vectors  $\{\frac{\partial}{\partial x_i}, \frac{\partial }{\partial y_i}\}$ (or dual vectors $\{dx_i,dy_i\}$) do not interact with each other, we are reduced to the case when $n=1$, and higher $n$ cases follow by taking tensor products.

When $n=1$, we have
$$\wedge^\star(\C \frac{\partial }{\partial x_1}\oplus  \C \frac{\partial }{\partial y_1})\cong \C 1 \oplus \C \frac{\partial }{\partial x_1}\oplus \C \frac{\partial }{\partial y_1}\oplus \C \frac{\partial }{\partial x_1}\wedge  \frac{\partial }{\partial y_1}.$$ 
Wedging with $\tau_p$ sends $1$ to $\frac{\partial}{\partial x_i} \wedge \frac{\partial }{\partial y_i}$ while contracting with $\omega_p$ takes $\frac{\partial }{\partial x_1}\wedge  \frac{\partial }{\partial y_1}$ back to $1$. Both operations kill the middle terms $ \frac{\partial }{\partial x_1} $ and $ \frac{\partial }{\partial y_1} $. The result follows.
\end{proof}

\begin{corollary}\label{cor-sl2-iso}
For any $j\in \{0,1,\dots, n\}$, there are isomorphisms of Hochschild cohomology groups  
\[
\tau^j\wedge(-): \mH^i(\Nt, \wedge^{n-j}T\Nt)^k\lra \mH^i(\Nt, \wedge^{n+j}T\Nt)^{k-2j}.
\]
\end{corollary}
\begin{proof}
Follows from the previous Theorem \ref{thm-sl2-action} and basic $\mathfrak{sl}_2$ representation theory.
\end{proof}

\begin{remark}
The statements of Theorem \ref{thm-sl2-action} and Corollary \ref{cor-sl2-iso} hold for many interesting holomorphic symplectic varieties that appear naturally in representation theory, such as Nakajima quiver varieties \cite{Na1}. The cotangent bundle of (partial) flag varieties in type A are special cases of quiver varieties. In the sequel \cite{LQ2}, we will discuss more examples in this family.
\end{remark}

Recall from \cite[Proposition 11]{BeLa} that $\zl_0:=\zl_0(u_q(\g))$ contains as a subalgebra two copies of the coinvariant algebra intersecting in a $1$-dimensional space.
The first copy sits in $\zl_0$ as the left most column in Table \ref{table-Hodge} and forms a genuine subalgebra. It is isomorphic to the pull-back of the usual cohomology ring of the flag variety $\mH^\bullet (X,\C)$ into the zeroth Hochschild cohomology of $\Nt$. The second copy, consisting of nilpotent elements, coincides with the bottom row of the formal Hodge diamond and constitutes the radical of $\zl_0$. The coinvariant algebra can be combinatorially defined as  
\begin{equation}\label{eqn-coinv-alg-type-A}
\mathrm{C}_\g:=\frac{S^\bullet (\h)}{S^\bullet(\h)_+^{W}},
\end{equation}
where $\h$ is a Cartan subalgebra in $\g$ and $W$ is the Weyl group of $\g$.

Using Theorem \ref{thm-sl2-action}, one obtains a larger subalgebra in the center, whose structure is relatively easy to describe.  The Poisson bivector field weaves  together these two copies of the coinvariant algebra. 
Let us set $\mathrm{deg}(\h)=1$ for the next statement.

\begin{corollary}\label{cor-big-subalg} The principal block of the center for the small quantum group $u_q(\g)$ contains the following subalgebra $\mathrm{\tau C}_\g$ generated by the coinvariant subalgebra $\mathrm{C}_\g$ and the Poisson bivector field $\tau$:
\[
\mathrm{\tau C}_\g:=\frac{\mathrm{C}_\g[\tau]}{\langle f\tau^k|f\in \mathrm{C}_\g, ~\mathrm{deg}(f)+k>\mathrm{dim}_\C (X)\rangle}.
\]
\end{corollary}
\begin{proof}
This follows from the injectivity of wedging with the appropriate powers of $\tau$ with elements lying on the first column of Table \ref{table-Hodge} (Corollary \ref{cor-sl2-iso}).
\end{proof}

From the work of Kerler \cite{Ker}, it is known that $ \mathrm{\tau C}_\g $ coincides with the entire principal block of the center when $\g=\mathfrak{sl}_2$. On the other hand, Theorem \ref{thm-sl3} shows that this subalgebra has codimension one in $\zl_0(\mathfrak{sl}_3)$.

\subsection{Conjectures} \label{sec-conj} 
In this section, we formulate several conjectures generalizing the case $\g = \mathfrak{sl}_3$.

The following conjecture comes from the explicit computations of several examples in type $A$, as well as some singular block computations which will appear in a sequel \cite{LQ2}.

\begin{conjecture}  \label{cong-triv-modules} Under the natural $SL_m(\C)$-action, the sheaf cohomology 
$
\mH^i(\Nt, \wedge^j T\Nt)^{-i-j} 
$
decomposes as a direct sum of trivial $SL_m(\C)$-modules for all $i,j \geq 0$ such that $i+j$ is even. 
\end{conjecture} 

Assuming Conjecture \ref{cong-triv-modules}, we would have the following statement 
that gives a purely algebraic description of   the principal block center $\zl_0$ of the small quantum $\mathfrak{sl}_m$. 

\begin{corollary} \label{cor-algebra-structure} In type $A$, there is an isomorphism of bigraded vector spaces
\[  
\zl_0 \cong \bigoplus_{i+j+k=0} \mH^i(\n, V_j^{k})^\h  , 
\]
where the $B$-modules 
\[ 
V_j = \wedge^j_{S^\bullet(u)}V_1 \quad \quad {\rm and} \quad \quad 
V_1:= \dfrac{S^\bullet(\lu) \otimes \g \oplus S^\bullet(\lu)\otimes \n }{\Delta(S^\bullet(\lu)\otimes \lb)} 
\]
are described in Theorem \ref{thm-equ-structure-TN} and Corollary \ref{cor-exterior-product}. 
The $k$-degrees of the components are as follows: $\op{deg}(\g) = \op{deg}(\lb) = 0$, $\op{deg}(\n) = -2$ 
and $\op{deg}(\lu) = 2$. 
\end{corollary}

Recall that Haiman's diagonal coinvariant algebra (see \cite{Hai}) in type $A_{m-1}$ is by definition
\begin{equation}\label{eqn-diag-coinv-alg}
\mathrm{DC}_m:=\dfrac{\C[x_1,\cdots, x_m, y_1,\dots, y_m]}{\C[x_1,\dots,x_m, y_1,\dots, y_m]^{S_m}_+}.
\end{equation}  
Any element $\sigma \in S_m$ acts on $\mathrm{DC}_m$ by simultaneously sending $x_i$ to $x_{\sigma(i)}$ and $y_i$ to $y_{\sigma(i)}$. Equip $\mathrm{DC}_m$ with the bidegree
\begin{equation}
\mathrm{deg}(x_i):=(1,0),\quad \mathrm{deg}(y_i):=(0,1)\quad (i=1,\dots, m).
\end{equation}
Then  each bigraded homogeneous component $\mathrm{DC}_m^{(i,j)}$ of bidegree $(i,j)$ is preserved under the symmetric group action. The usual coinvariant algebra $\mathrm{C}_m:=\mathrm{C}_{\mathfrak{sl}_m}$ can be recovered from $\mathrm{DC}_m$ by specializing either the $x$ or $y$ to zero. 
   
\begin{conjecture}\label{conj-DC}
In type $A$, the center of the principal block $\zl_0:=\zl_0(\mathfrak{sl}_m) $ for the small quantum group $u_q(\mathfrak{sl}_m)$ is isomorphic, up to a bigrading transformation, to the diagonal coinvariant algebra as a bigraded vector space. More precisely:
\begin{itemize}
\item[(1)] There exists a symmetric group $S_m$ action on $\zl_0$, extending the action of $S_m$ on the coinvariant subalgebra $\mathrm{C}_m\subset \zl_0$.
\item[(2)] The symmetric group action commutes with the action of $\mathfrak{sl}_2$ constructed in Theorem \ref{thm-sl2-action}.
\item[(3)] As a bigraded representation of $S_m$, there is an isomorphism of representations
\[
\zl_0\cong \mathrm{DC}_m\otimes \mathrm{sgn},
\]
where $\mathrm{sgn}$ stands for the $1$-dimensional sign representation of the symmetric group sitting in bidegree $(0,0)$. The bigradings are matched as follows: for any $i,j\in \N$, 
\[
\zl_0^{i,j}=\mH^{i}(\Nt,\wedge^jT\Nt)^{-i-j}\cong \left(\mathrm{DC}_m\otimes \mathrm{sgn}\right)^{{m\choose 2}-\frac{i+j}{2},\frac{j-i}{2}},
\]
where ${m \choose 2}$ is the complex dimension of the flag variety $X=SL_m(\C)/B$.
\end{itemize}
In particular, the dimension of the principal block of the  center is equal to
$$\mathrm{dim}(\zl_0(\mathfrak{sl}_m))=\mathrm{dim}(\mathrm{DC}_m)=(m+1)^{m-1}.$$
\end{conjecture} 

Below we list a few further remarks about the conjecture.

\begin{remark}
\begin{enumerate}
\item[(i)] The expected symmetric group action in Conjecture \ref{conj-DC} should come from taking the zeroth Hochschild cohomology $\mathrm{HH}^0$ of the Steinberg variety $\mathcal{Z}=\Nt\times_\mathcal{N}\Nt$, even though $\mathcal{Z}$ is singular and $\mathrm{HH}^0$ needs to be treated more carefully. By the result in \cite{Riche}, the convolution with the structure sheaves of components of $\mathcal{Z}$ gives rise to a braid group action on the $\C^*$-equivariant derived category of coherent sheaves on $\Nt$. We expect that, on the level of zeroth Hochschild cohomology, the action should factor through the symmetric group $S_m$.
The second part of the conjecture is then a consequence of the fact that the subvarieties $\mathcal{Z}_w$ are Lagrangian. This implies that the natural symplectic form $(\omega, -\omega)$ on $\Nt\times \Nt$, and the natural Poisson bivector field $(\tau,-\tau)$ vanish on $\mathcal{Z}_w$.  
\item[(ii)] There are two reasons for the sign character to appear in part (3) of the conjecture. The first reason is that, by inspection, the diagonal in the formal Hodge diamond is spanned by powers of the Poisson bivector field. These forms are invariant under the symmetric group action by part (1) of the conjecture. On the other hand, the diagonals in Haiman's coinvariant algebra consist of sign representations of the symmetric group. The second reason is that the principal block of the big quantum group at a root of unity categorifies the \emph{antispherical module} of the affine Hecke algebra: this is exactly the induced module of the affine Hecke algebra from the sign character of the finite Hecke subalgebra. A recent work of Riche and Williamson \cite{RW} has given a categorical explanation of this phenomenon in type $A$ for algebraic groups over finite characteristic fields via categorification. Their method also applies to (big) quantum groups at roots of unity. 
\item[(iii)] The last statement concerning the dimension of the center is a consequence of Haiman's character formula for $\mathrm{DC}_n$ in his proof of the $n!$ Theorem \cite{Hai2}. Other proofs of the character formula which generalize beyond type $A$ are given by Gordon in \cite{Gor} and Cherednik \cite{CheDiagonal} via representation theory of \emph{double affine Hecke algebras} (DAHA).  It is also a natural question to ask whether the symmetric group action and the $\mathfrak{sl}_2$ action can be integrated into a DAHA action on $\zl_0(\mathfrak{sl}_m)$, as done by Gordon and Cherednik. This question may be closely related to part (i) since the DAHA could be possibly realized as a deformed generalized cohomology theory of the Steinberg variety.
\item[(iv)] It is known that the commutative algebra structures on these vector spaces disagree. The socle of $\zl_0(\g)$ contains a copy of the coinvariant algebra $\mathrm{C}_\g=S^\bullet(\h)/S^\bullet (\h)^{W}_+$,  
which is identified with the horizontal bottom row in the formal Hodge diamond \eqref{table-Hodge}. In particular, when $m=3$, 
the socle of $\zl_0(\mathfrak{sl}_3)$ is at least six dimensional, while  the socle of $\mathrm{DC}_3$ is five dimensional.
\end{enumerate}
\end{remark}

\subsection{Further evidence}\label{sec-further-evidence}
In this section, we compute the center of the small quantum group  $u_q(\mathfrak{sl}_4)$. Let $G=SL_4(\C)$ and $B$ be its Borel subgroup of lower triangular matrices. The following result confirms the conjectures made in Section \ref{sec-conj}.

\begin{theorem}\label{conj-sl4}
The bigraded formal Hodge diamond for the principal block of small quantum $\mathfrak{sl}_4$ at a root of unity is given by
\begin{equation}\label{table-Hodge-sl4}
\begin{gathered}
\begin{array}{|c||c|c|c|c|c|c|c|} 
\hline
 {\scriptstyle i+j=0}        & 1          &            &            &            &            &            & \\ \hline 
 {\scriptstyle i+j=2}        & 3          &  1         &            &            &            &            & \\ \hline 
 {\scriptstyle i+j=4}        & 5          &  4         &  1         &            &            &            & \\ \hline
 {\scriptstyle i+j=6}        & 6          &  9         &  4         &  1         &            &            & \\ \hline 
 {\scriptstyle i+j=8}        & 5          &  11        & 9          &  4         &  1         &            & \\ \hline
 {\scriptstyle i+j=10}       & 3          &  8         & 11         &  9         &  4         &  1         & \\ \hline 
 {\scriptstyle i+j=12}       & 1          &  3         & 5          &  6         &  5         &  3         &  1 \\ \hline \hline
 {\scriptstyle h^{i,j}} & {\scriptstyle j-i=0} & {\scriptstyle j-i=2} & {\scriptstyle j-i=4} & {\scriptstyle j-i=6} & {\scriptstyle j-i=8} & {\scriptstyle j-i=10} & {\scriptstyle j-i=12} \\  \hline
\end{array}  
\end{gathered} \ ,
\end{equation}
where $h^{i,j}=\mathrm{dim}(\mH^i(\Nt,\wedge^jT\Nt)^{-i-j})$. Furthermore, the space $\mH^i(\Nt,\wedge^jT\Nt)^{-i-j}$ consists of trivial $G$-modules for each pair of $(i,j)$ appearing in the table.
\end{theorem}

The proof consists of several steps. First, we have the following statement. 

\begin{proposition}\label{prop-diagonal}
The diagonal entries in Table \ref{table-Hodge-sl4} each consist of one copy of the trivial $\mathfrak{sl}_4$-representation:
\[
\mH^0(\Nt,\wedge^{2r}T\Nt)^{-2r}\cong L_0.
\]
\end{proposition}

\begin{proof} 
The result can be computed explicitly using the algorithm developed in Section \ref{sec-Springer}. A proof of a more general 
statement valid for any $\Nt_P:=T^*(G/P)$, where $G$ is a simple complex Lie group and  $P$ is a parabolic subgroup, will be given in the sequel \cite{LQ2}, via some basic properties of (semi)stable vector bundles.
\end{proof} 

\begin{proof}[Proof of Theorem \ref{conj-sl4}]
Taking into account Proposition \ref{prop-diagonal} and Corollary \ref{cor-big-subalg}, the entries of the first column, bottom row and the main diagonal of Table \ref{table-Hodge-sl4} coincide with the respective values in Haiman's diagonal coinvariant algebra. 

To compute the remaining entries, we use the following three steps:

\begin{enumerate}
\item[(i)] Consider the cohomolgy groups for the natural subquotient vector bundles of 
$$\mathrm{pr}_*(\wedge^k T\Nt)^{-2r}\cong G\times_B V_k^{-2r} $$ and show that in the appropriate cohomological degree $\mH^{2r-k}$, the bundles only contribute copies of trivial $G$-modules to the total cohomology.
\item[(ii)] Compute the terms in the second left-most column using the relative Lie algebra cohomology for the corresponding sheaves. By Step (i) and Proposition \ref{rel_lie_BGG}, this is reduced to  computing the multiplicity space of trivial representation in the corresponding Hochschild cohomology term. Then Proposition \ref{rel_lie_BGG} and 
Corollary \ref{cor-exterior-product} provide an explicit algebraic algorithm for this computation. 
The terms $h^{1,3} =4$ and $h^{4,6} =8$ have been computed by hand; the terms $h^{2,4} =9$ and $h^{3,5}=11$ 
have been computed using the same algorithm implemented in Python. 

\item[(iii)] Give an upper bound for $h^{1,5}$ and $h^{2,6}$ by analyzing the cohomology of the subquotient sheaves. 
The upper bound can be directly read off from the dimensions of the cohomology groups considered in Step (i). 
A Python-based computation gives $h^{1,5} \leq 4 = h^{1,3}$ and $h^{2,6} \leq 9 = h^{2,4}$.  
Then the $\mathfrak{sl}_2$-action  along the diagonals (Theorem \ref{thm-sl2-action}) assures that this upper bound is always achieved for all remaining entries along the same diagonal. 
\end{enumerate}
\end{proof} 

\begin{remark} 
In Steps (i) and (ii) above we will need explicit expressions for the maps in the BGG resolution complex for the trivial module for $\mathfrak{sl}_4$. Since we were unable to find them in the literature, and because they may present an independent interest from a representation-theoretic viewpoint, we will record these maps below.  
 Denote by $\n$ the space of strictly lower triangular matrices, and $U(\n)$ its universal enveloping algebra. This is the associative algebra generated by the Chevalley generators $f_1$, $f_2$ and $f_3$ subject to the Serre relations
\[
f_1f_3-f_3f_1=0,
\]
\[
f_1^2f_2-2f_1f_2f_1+f_2f_1^2=0, \quad \quad f_2^2f_1-2f_2f_1f_2+f_1f_2^2=0,
\]
\[
f_2^2f_3-2f_2f_3f_2+f_3f_2^2=0, \quad \quad f_3^2f_2-2f_3f_2f_3+f_2f_3^2=0.
\]
The BGG complex for the zero weight looks as follows
\begin{equation}
M_\bullet: 0\lra M_6 \stackrel{d_5}{\lra}  M_5 \stackrel{d_4}{\lra}  M_4 \stackrel{d_3}{\lra} M_3 \stackrel{d_2}{\lra}  M_2 \stackrel{d_1}{\lra}  M_1 \stackrel{d_0}{\lra}  M_0 \lra 0,
\end{equation}
where each $M_i$ is a direct sum of free $\h$-graded $U(\n)$-modules, and $d_i$ are $\h$-grading preserving maps. To describe the modules and the differentials, we will use the following notation: if $s_{i_1}s_{i_2}\cdots s_{i_k}$ is a product of simple reflections in $S_3$, then we will abbreviate
\begin{equation}
U_{i_1i_2 \dots i_k}:=U(\n)\cdot 1_{(s_{i_1}s_{i_2}\cdots s_{i_k})\cdot 0},
\end{equation} 
the right-hand side denoting the free $U(\n)$-module generated by an $\h$-weight vector of weight $(s_{i_1}s_{i_2}\cdots s_{i_k})\cdot 0$. Also write $U_0=U(\n)$ and $U_{w_0}=U(\n)\cdot 1_{w_0\cdot 0}=U_{123121}$. 

Then the modules $M_0,\dots, M_6$ can be explicitly identified with
\[
M_0\cong U_0, \quad M_6\cong U_{w_0},
\]
\[
M_1\cong U_1\oplus U_2\oplus U_3, \quad M_5 \cong U_{23121} \oplus U_{12321}\oplus U_{21232},
\]
\[
M_2 \cong U_{21}\oplus U_{12}\oplus U_{31}\oplus U_{32}\oplus U_{23},\quad
M_4 \cong U_{1321}\oplus U_{2321} \oplus U_{1231} \oplus U_{2312} \oplus U_{1232},
\]
\[
M_3\cong U_{121}\oplus U_{321}\oplus U_{231}\oplus U_{312}\oplus U_{123}\oplus U_{232}.
\]

The differentials $d_0,\dots, d_5$ can be written as matrices with coefficients in $U(\n)$, and they act by \emph{right multiplication} on the free modules. For instance, we write
\begin{equation}\label{eqn-d0}
d_0 = 
\left(
\begin{array}{c|c|c||c}
{\scriptstyle 1} & {\scriptstyle 2} & {\scriptstyle 3}  & \\ \hline \hline
f_1 & f_2 & f_3 & {\scriptstyle 0} 
\end{array}
\right)
\end{equation}
to indicate that, for any $x,y,z\in U(\n)$,
\[
d_0: U_1\oplus U_2\oplus U_3\lra U_0,\quad (x1_{s_1\cdot 0}, \; y1_{s_2\cdot 0}, \; z1_{s_3\cdot 0})\mapsto xf_1+yf_2+zf_3.
\]
In this notation, we identify $d_1$ with the matrix
\begin{equation}\label{eqn-d1}
\left(
\begin{array}{c|c|c|c|c||c}
{\scriptstyle 21}  & {\scriptstyle 12} & {\scriptstyle 31} & {\scriptstyle 32} & {\scriptstyle 23} & \\ \hline \hline
\scriptstyle{-f_2^2}& \scriptstyle{2f_1f_2-f_2f_1}  
                                       & \scriptstyle{-f_3}&                   &                   &{\scriptstyle 1}\\ 
\hline
\scriptstyle{2f_2f_1-f_1f_2}   
                   &\scriptstyle{-f_1^2}&                  &\scriptstyle{f_3^2}&\scriptstyle{f_3f_2-2f_2f_3 }                           &{\scriptstyle 2}\\
\hline
                  &                   & \scriptstyle{f_1}  &   \scriptstyle{f_2f_3-2f_3f_2}  & \scriptstyle{f_2^2}      &{\scriptstyle 3}  
\end{array}
\right),
\end{equation}
 $d_2$ with the matrix
\begin{equation}\label{eqn-d2}
\left(
\scalemath{0.9}{
\begin{array}{c|c|c|c|c|c||c}
{\scriptstyle 121} & {\scriptstyle 321} & {\scriptstyle 231} & {\scriptstyle 312} & {\scriptstyle 123} & {\scriptstyle 232} &                  \\ \hline\hline
\scriptstyle{-f_1} &\scriptstyle{f_3^3} &\scriptstyle{3f_2f_3-2f_3f_2}  &                    &                    &                      
     & {\scriptstyle 21} \\ \hline
\scriptstyle{-f_2} &                    &                    & \scriptstyle{f_3^2} & \begin{array}{c}\scriptstyle{6f_1f_2f_3-4f_2f_1f_3}\\\scriptstyle{-3f_1f_3f_2+2f_3f_2f_1}\end{array} &                                                                         
     & {\scriptstyle 12} \\  \hline  
                   &  \begin{array}{c}\scriptstyle{-f_3^2 f_2^2 - 4 f_3f_2f_3f_2}\\\scriptstyle{-2f_2f_3f_2f_3+6f_3f_2^2f_3} \end{array}                  &\scriptstyle{-f_2^3}& \begin{array}{c}\scriptstyle{4f_1f_3f_2-2f_3f_2f_1}\\\scriptstyle{-2f_1f_2f_3+f_2f_3f_1} \end{array}      &                    
 \begin{array}{c}\scriptstyle{f_1^2f_2^2+4f_1f_2f_1f_2}\\\scriptstyle{+2f_2f_1f_2f_1-6f_1f_2^2f_1} \end{array}                  
                   &                                                                         
     & {\scriptstyle 31} \\  \hline 
                   & \begin{array}{c}\scriptstyle{-6f_3f_2f_1+4f_2f_1f_3}\\ \scriptstyle{+3f_1f_3f_2-2f_1f_2f_3}\end{array}                             &                    &        \scriptstyle{f_1^2}       &                    &                                                   \scriptstyle{f_2} & {\scriptstyle 32} \\  \hline
                   &                    &\scriptstyle{3f_2f_1-2f_1f_2} 
                                                             &                    & \scriptstyle{-f_1^3}                                                              &\scriptstyle{f_3}
& {\scriptstyle 23} \\          
\end{array}}
\right),
\end{equation}
$d_3$ with the matrix
\begin{equation}\label{eqn-d3}
\left(
\scalemath{0.9}{
\begin{array}{c|c|c|c|c||c}
 \scriptstyle{1321} & \scriptstyle{2321} & \scriptstyle{1231} & \scriptstyle{2312} & \scriptstyle{1232} &  \\ 
 \hline \hline
\scriptstyle{-f_3^3}&                    &  \begin{array}{c}
\scriptstyle{6f_1f_2f_3-4f_1f_3f_2}\\ \scriptstyle{-3f_2f_1f_3+2f_3f_2f_1}
\end{array}                 &   \begin{array}{c}\scriptstyle{f_2^2 f_3^2 + 4 f_2f_3f_2f_3}\\\scriptstyle{+2f_3f_2f_3f_2-6f_2f_3^2f_2} \end{array}                 &                &  \scriptstyle{121}\\ \hline 
\scriptstyle{-f_1}  & \scriptstyle{f_2}  &                   &                  &                &  \scriptstyle{321}\\ \hline
                    &\scriptstyle{-f_3^2}&\scriptstyle{f_1^2}& \begin{array}{c}
\scriptstyle{4f_2f_1f_3-2f_1f_2f_3}\\ \scriptstyle{-2f_3f_2f_1+f_1f_3f_2}
\end{array}                                      &                &  \scriptstyle{231}\\ \hline
\begin{array}{c}{\scriptstyle{2f_2f_3}}\\ \scriptstyle{-3f_3f_2}\end{array}&           &                   & \scriptstyle{f_2^3} & \begin{array}{c}\scriptstyle{2f_2f_1}\\ \scriptstyle{-3f_1f_2}\end{array}    &  \scriptstyle{312}\\ \hline
                    &                    & \scriptstyle{f_2} &                  &  \scriptstyle{f_3}    &  \scriptstyle{123}\\ \hline 
                    &\begin{array}{c}
\scriptstyle{6f_3f_2f_1-4f_1f_3f_2}\\ \scriptstyle{-3f_2f_1f_3+2f_1f_2f_3}
\end{array}                    &                   &   \begin{array}{c}\scriptstyle{-f_2^2f_1^2-4f_2f_1f_2f_1}\\\scriptstyle{-2f_1f_2f_1f_2+6f_2f_1^2f_2} \end{array}                      & \scriptstyle{f_1^3}&  \scriptstyle{232} \\ 
\end{array} }
\right),
\end{equation}
$d_4$ with the matrix
\begin{equation}\label{eqn-d4}
\left(
\begin{array}{c|c|c||c}
 \scriptstyle{23121}         & \scriptstyle{12321}          &      \scriptstyle{21232}         &                   
\\ \hline \hline
 \scriptstyle{f_2^2}         & \scriptstyle{f_2f_1-2f_1f_2} &                                  &  \scriptstyle{1321} 
\\ \hline
\scriptstyle{2f_2f_1-f_1f_2} & \scriptstyle{-f_1^2}         &                                  &  \scriptstyle{2321} 
\\ \hline
                             & \scriptstyle{-f_3^2}         &  \scriptstyle{f_3f_2-2f_2f_3}    &  \scriptstyle{1231} 
\\ \hline
\scriptstyle{f_3}            &           &         \scriptstyle{f_1}                           &  \scriptstyle{2312} 
\\ \hline
                             &\scriptstyle{2f_3f_2-f_2f_3}  &  \scriptstyle{f_2^2}             &  \scriptstyle{1232} 
\\ 
\end{array}
\right),
\end{equation}
and, finally, $d_5$ with the matrix
\begin{equation}\label{eqn-d5}
\left(
\begin{array}{c||c}
        w_0            &                     \\ \hline \hline
        -f_1           & \scriptstyle{23121} \\ \hline
        -f_2           & \scriptstyle{12321} \\ \hline
         f_3           & \scriptstyle{21232} \\ 
\end{array}
\right) .
\end{equation}

It follows from Proposition \ref{rel_lie_BGG} that, if $E$ is any $B$-module, then the multiplicity of $L_0$ inside the cohomology group
$ \mH^\bullet(X,G\times_B E) $ can be computed as the dimension of the cohomology of the complex $(\mathrm{Hom}_{\n}(M_\bullet, E)^\h,d_\bullet^*)$, which now takes the form
\begin{equation}
0\lra E[0]\to \cdots \to \bigoplus_{l(w)=j} E[w\cdot 0]\stackrel{d_j^*}{\lra}\bigoplus_{l(w)=j+1} E[w\cdot 0]\to \cdots \to E[w_0\cdot 0]\lra 0,
\end{equation}
where the differentials $d_j^*$ ($j=0,\dots ,5$) are obtained by letting the above matrices of lowering operators \eqref{eqn-d0}--\eqref{eqn-d5} act on the corresponding weight spaces of $E$.

Using the BGG resolution maps described above, we have obtained the result confirming the Conjecture \ref{conj-DC} 
in case of $u_q(\mathfrak{sl}_4)$.
\end{remark}

\addcontentsline{toc}{section}{References}



%

\vspace{0.1in}

\noindent A.~L.: { \sl \small Department of Mathematics, \'{E}cole Polytechnique F\'{e}d\'{e}rale de Lausanne,  SB MATHGEOM, MA A2 407 (B\^atiment MA), Station 8, CH-1015 Lausanne, Switzerland} \newline \noindent {\tt \small email: anna.lachowska@epfl.ch}

\vspace{0.1in}

\noindent Y.~Q.: { \sl \small Department of Mathematics, Yale University, New Haven, CT 06511, USA} \newline \noindent {\tt \small email: you.qi@yale.edu}

%

\end{document}